\documentclass[a4paper,12pt]{amsart}

\usepackage{amssymb}
\usepackage{latexsym}
\usepackage{amsmath}
\usepackage{amscd}
\usepackage{graphicx}

\title[Abelianizations of derivation Lie algebras]
{Abelianizations of derivation Lie algebras 
of the free associative algebra and the free Lie algebra
}

\author[Morita]{Shigeyuki Morita}
\email{morita@ms.u-tokyo.ac.jp}
\author[Sakasai]{Takuya Sakasai}
\address{Graduate School of Mathematical Sciences, 
The University of Tokyo, 
3-8-1 Komaba Meguro-ku Tokyo 153-8914, Japan}
\email{sakasai@ms.u-tokyo.ac.jp}
\author[Suzuki]{Masaaki Suzuki}
\address{Department of Mathematics, 
Akita University, 
1-1 Tegata-Gakuenmachi, Akita, 010-8502, Japan}
\email{macky@math.akita-u.ac.jp}


\newtheorem{thm}{Theorem}[section]

\newtheorem{lem}[thm]{Lemma}
\newtheorem{cor}[thm]{Corollary}
\theoremstyle{definition}
\newtheorem{definition}[thm]{Definition}

\newtheorem{remark}[thm]{Remark}

\newtheorem{conj}[thm]{Conjecture}

\begin{document}

\newcommand{\Mg}{\mathcal{M}_g}
\newcommand{\Mgp}{\mathcal{M}_{g,\ast}}
\newcommand{\Mgb}{\mathcal{M}_{g,1}}

\newcommand{\hg}{\mathfrak{h}_g}
\newcommand{\ag}{\mathfrak{a}_g}
\newcommand{\Ln}{\mathcal{L}_n}

\newcommand{\Sg}{\Sigma_g}
\newcommand{\Sgb}{\Sigma_{g,1}}

\newcommand{\Symp}[1]{Sp(2g,\mathbb{#1})}
\newcommand{\symp}[1]{\mathfrak{sp}(2g,\mathbb{#1})}
\newcommand{\gl}[1]{\mathfrak{gl}(n,\mathbb{#1})}

\newcommand{\At}[1]{\mathcal{A}_{#1}^t (H)}
\newcommand{\Hq}{H_{\mathbb{Q}}}

\newcommand{\Ker}{\mathop{\mathrm{Ker}}\nolimits}
\newcommand{\Hom}{\mathop{\mathrm{Hom}}\nolimits}
\renewcommand{\Im}{\mathop{\mathrm{Im}}\nolimits}

\newcommand{\Der}{\mathop{\mathrm{Der}}\nolimits}
\newcommand{\Out}{\mathop{\mathrm{Out}}\nolimits}
\newcommand{\Aut}{\mathop{\mathrm{Aut}}\nolimits}
\newcommand{\Q}{\mathbb{Q}}
\newcommand{\Z}{\mathbb{Z}}

\begin{abstract}
We determine the abelianizations of 
the following three kinds of graded Lie algebras
in certain stable ranges: 
derivations of the free associative algebra, 
derivations of the free Lie algebra and 
symplectic derivations of the free associative algebra. 
In each case, we consider both the whole derivation Lie algebra
and its ideal consisting of derivations with {\it positive} degrees.
As an application of the last case, 
and by making use of a theorem of Kontsevich,
we obtain a
new proof of the
vanishing theorem of Harer concerning the
top rational cohomology group of the mapping class group
with respect to its virtual cohomological dimension.
\end{abstract}

\renewcommand\baselinestretch{1.1}
\setlength{\baselineskip}{16pt}

\newcounter{fig}
\setcounter{fig}{0}

\maketitle

%
%\input{Section1and2.tex}
%
%\input{Section3and4.tex}
%
%\input{Section5.tex}
%
%\input{Section6and7andRef.tex}

%-----------------introduction.tex--------2012/3/5, Sakasai ----
\section{Introduction and statements of the main results}\label{sec:intro}

In this paper, we consider graded Lie algebras,
over $\Z$, consisting
of derivations of free associative or free Lie algebras
generated by a free abelian group of finite rank.
We also consider the cases where the rank 
%$\mathrm{rank}\; V$ 
is even and 
equipped with a non-degenerate skew-symmetric bilinear form.
In this case, we consider the graded Lie algebras consisting
of {\it symplectic} derivations. We also consider the rational
forms of them.
These Lie algebras appear 
naturally in various aspects of topology and it should be
an important problem to analyze the structure of them.

To be more precise, let $H_n$ denote the free abelian group of
rank $n$ generated by $x_1,x_2, \ldots,x_n$
and let $T(H_n), \mathcal{L}_n$ be the free associative algebra 
without constant terms and the free Lie algebra generated by $H_n$, respectively. 
We denote by $\Der (T(H_n))$ and 
$\Der (\mathcal{L}_n)$ the graded Lie algebras consisting of
derivations of $T(H_n)$ and $\mathcal{L}_n$.
In the case where $n=2g$ and $H_{2g}\otimes\Q$
is equipped with a skew-symmetric bilinear form
so as to be identified with the standard symplectic
vector space of dimension $2g$, 
we denote by
$\mathfrak{a}_g$ and $\mathfrak{h}_{g,1}$ the Lie subalgebras
of $\Der (T(H_{2g}\otimes\Q))$ and 
$\Der (\mathcal{L}_{2g}\otimes\Q)$ consisting of {\it symplectic} derivations, 
respectively. 
See Sections \ref{sec:associative}, 
\ref{sec:lie}, \ref{sec:symp} 
for detailed definitions.

Our main result concerns the abelianizations of the above Lie algebras
as well as certain ideals of them in certain stable ranges.
The natural inclusion $H_n\subset H_{n+1}$ induces 
a sequence
$$
\Der (\mathcal{L}_1)\subset\Der (\mathcal{L}_2)\subset
\cdots\subset \Der (\mathcal{L}_\infty)
$$
of embeddings of Lie algebras where the last Lie algebra
denotes the union of the preceding ones. We also consider
similar series for the other Lie algebras. The abelianization 
$\Der (\mathcal{L}_\infty)/[\Der (\mathcal{L}_\infty),\Der (\mathcal{L}_\infty)]$ 
of the limit algebra, denoted by $H_1(\Der (\mathcal{L}_\infty))$,
is nothing other than the direct limit
$\lim_{n\to\infty} H_1(\Der (\mathcal{L}_n))$
of the abelianization of each member of the above series
and we call this the {\it stable} abelianization.

Now our main result is the first and the third cases of the following theorem which determines the stable abelianization of the three Lie algebras. 
The second statement follows from Theorem \ref{thm:lie} 
which gives a slight improvement of a beautiful 
work of Kassabov \cite[Theorem 1.4.11]{kassabov} 
and our proof is very close to the original one. 

\begin{thm}\label{thm:h1}
The stable abelianizations of the three Lie algebras
$\Der (T(H_n)), \Der (\mathcal{L}_n), \mathfrak{a}_g$
are given as follows.
\begin{align*}
&\mathrm{(i)} \ \lim_{n\to\infty} H_1(\Der (T(H_n))\cong \Z\\
&\mathrm{(ii)} \ \lim_{n\to\infty} H_1(\Der (\mathcal{L}_n))\cong \Z\\
&\mathrm{(iii)} \ \lim_{g\to\infty} H_1(\mathfrak{a}_g)=0.
\end{align*}
\end{thm}

Let $\Der^+ (T(H_n))$ and $\Der^+ (\mathcal{L}_n)$
denote the ideals of $\Der(T(H_n))$ and $\Der(\mathcal{L}_n)$
consisting of derivations of {\it positive} degrees. 
Similarly we denote by 
$\mathfrak{a}^+_g\subset \mathfrak{a}_g$ and $\mathfrak{h}^+_{g,1}\subset \mathfrak{h}_{g,1}$
the ideals consisting of derivations of positive degrees.
The proof of Theorem \ref{thm:h1} is based on careful studies
of the bracket operations in these ideals. 
We can summarize our results on the structures
of these ideals as follows (see more precise
statements in Sections \ref{sec:associative}, 
\ref{sec:lie}, \ref{sec:symp}).

\begin{thm}\label{thm:h1+}
The Lie algebras $\Der^+(T(H_n))$ and $\mathfrak{a}^+_g$ are
``finitely generated" in certain stable ranges. 
More precisely we have the following.
\begin{align*}
& \mathrm{(i)} \ \text{Up to degree $n-1$, $\Der^+(T(H_n))$
is generated by the degree $1$ part $H_n^*\otimes H_n^{\otimes 2}$}
\\
& \hspace{4cm}    \text{together with a certain summand $H_n^{\otimes 2}$ of degree $2$ }\\
& \mathrm{(ii)}\ \text{Up to degree $g$, $\mathfrak{a}_g^+$
is generated by the degree $1$ part $S^3 H_\Q\oplus \wedge^3 H_\Q$}\\
& \hspace{4cm}    \text{together with a certain summand $\wedge^2 H_\Q/\langle\omega_0\rangle$ of degree $2$ }
\end{align*}
%where we denote simply by $H_\Q$ the symplectic vector space
%$H_{2g}\otimes\Q$  and $S^3 H_\Q$ and $\wedge^3 H_\Q$ denote
%the third symmetric and exterior powers of $H_\Q=H_{2g}\otimes\Q$ respectively.
where $S^3 H_\Q$ and $\wedge^3 H_\Q$ denote
the third symmetric and exterior powers of 
the symplectic vector space $H_\Q=H_{2g}\otimes\Q$ respectively.
Also $\langle\omega_0\rangle$ denotes the submodule
of the second exterior power $\wedge^2 H_\Q$ spanned by the
symplectic class.
\end{thm}

The important point here is that the numbers of the generating summands
are {\it independent} of $n$ and $g$ whereas the stable ranges
grow linearly with respect to them. We mention that it is still
unknown whether the above ideals are finitely generated in the
usual sense or not.

In a sharp contrast with the above result, the Lie algebras
$\Der^+ (\mathcal{L}_n)$ and $\mathfrak{h}^+_{g,1}$  are known to be {\it not}
finitely generated. In fact, the degree $1$ part and the trace maps introduced in \cite{morita93}
define {\it surjective} homomorphisms
\begin{align*}
&\Der^+ (\mathcal{L}_n) \longrightarrow 
(H_n^*\otimes \wedge^2 H_n) \oplus \bigoplus_{k=2}^\infty S^{k} H_n\\
&\mathfrak{h}^+_{g,1} \longrightarrow 
\wedge^3 H_\Q \oplus \bigoplus_{k=1}^\infty S^{2k+1} H_\Q
\end{align*}
of Lie algebras where the targets are understood
to be {\it abelian} Lie algebras.

A theorem of Kassabov cited above implies that the upper homomorphism induces an isomorphism in the first {\it rational} homology group $H_1(\ \ ;\Q)$ of Lie algebras
in a certain stable range. Our Theorem \ref{thm:lie} implies the same statement 
with respect to the first {\it integral} homology group but with a smaller stable range. 
%but with the first {\it integral} homology group.
The first author once conjectured that the lower homomorphism would also induce an isomorphism in $H_1$. However, very recently Conant, Kassabov and Vogtmann \cite{ckv} proved that this is not the case, indicating that the Lie algebra
$\mathfrak{h}_{g,1}$ has a truly deep structure. 
Nevertheless, in view of known results
together with numbers of explicit computations we have made
so far, it seems still reasonable to make the following.

\begin{conj}\label{conj:h}
The stable abelianization of the Lie algebra
$\mathfrak{h}_{g,1}$ vanishes. Namely
$$
\lim_{g\to\infty} H_1(\mathfrak{h}_{g,1})=0.
$$
\end{conj}

The Lie algebra $\mathfrak{a}_g$ was introduced by Kontsevich in \cite{kontsevich1, kontsevich2}. It is one of the three Lie algebras considered in his theory of graph homology. One of the 
other Lie algebras, denoted $\ell_g$ by him, is the same as
$\mathfrak{h}_{g,1}$ which appeared already in the theory
of Johnson homomorphisms of the mapping class groups
both in the contexts of topology and number theory. 
Furthermore this Lie algebra is defined over $\Z$ rather than $\Q$
and the integral structure should be important in both contexts.

Kontsevich proved a remarkable theorem which gives close relations between the stable homology of $\mathfrak{a}_g$ and $\mathfrak{h}_{g,1}$ with the totalities
of the rational cohomology groups of the mapping class groups
(see Theorem \ref{thm:vanishing}), and those of the outer automorphism groups $\Out F_n$ of free groups $F_n\ (n\geq 2)$,
respectively.

If we combine Theorem \ref{thm:h1} with the former case of
this theorem of Kontsevich, we obtain a new proof 
of the following vanishing result
of Harer for the top rational cohomology group of the mapping class group with respect to its virtual cohomological dimension which was also 
determined by Harer \cite{harer}.

\begin{thm}[Harer \cite{harer2}]\label{thm:harerv}
For any $g \ge 2$, the top degree
rational cohomology group
of the mapping class group $\mathcal{M}_g$,
with respect to its virtual cohomological dimension,
vanishes. Namely
$$
H^{4g-5}(\mathcal{M}_g;\Q)=0 \quad (g\geq 2).
$$
\end{thm}

See Theorem \ref{thm:vanishing} for details.
We have heard that Church, Farb and Putman have also proved
the above vanishing theorem in their recent work (see \cite{cfp}).

\begin{remark}
We can deduce from the latter case of the
theorem of Kontsevich mentioned above that Conjecture \ref{conj:h} is equivalent to the statement that the top rational cohomology group 
$H^{2n-3}(\Out F_n;\Q)$ vanishes for any $n\geq 2$ with respect to its virtual cohomological dimension which was determined by Culler and Vogtmann \cite{cuv}.
\end{remark}

{\it Acknowledgement} We would like to
thank John Harer for informing us about his vanishing 
theorem and also for pointing out a possible relation to
a recent work of Church, Farb and Putman mentioned above. 
Thanks are also due to the referees for helpful suggestions. 

The authors were partially supported by KAKENHI (No.~24740040 and 
No.~24740035), 
Japan Society for the Promotion of Science, 
Japan.

%-----------------homology5a.tex--------2012/3/28, MSS----
\section{Lie algebra and its homology}\label{sec:homology}

We begin by recalling a few basic facts from the theory of Lie algebras and 
their homology groups. 

\begin{definition}
A vector space $\mathfrak{g}$ over 
$\mathbb{Q}$, is called a {\it Lie algebra} 
if it has a $\mathbb{Q}$-bilinear map 
\[[\,\cdot\,,\,\cdot\,]: \mathfrak{g} \otimes \mathfrak{g} 
\longrightarrow \mathfrak{g},\]
which is called the {\it bracket} map, satisfying the following two conditions:
\begin{itemize}
\item (anti-symmetry) $[x,y]=-[y,x]$ holds for any $x,y \in \mathfrak{g}$; and 
\item (Jacobi identity) 
$[x,[y,z]]+[y,[z,x]]+[z,[x,y]]=0$ holds for any $x,y,z \in \mathfrak{g}$.
\end{itemize}
If we replace a vector space and $\mathbb{Q}$-bilinear map,
in the above definition, by an abelian group and $\mathbb{Z}$-bilinear map
respectively, then we obtain the concept of the Lie algebra over $\Z$.
\end{definition}

The image $[\mathfrak{g}, \mathfrak{g}]$ of the bracket map is 
an ideal of 
$\mathfrak{g}$. 

\begin{definition}\label{def:abelianization}
For a Lie algebra $\mathfrak{g}$, the quotient vector space 
\[H_1 (\mathfrak{g}):= \mathfrak{g}/[\mathfrak{g}, \mathfrak{g}]\]
considered as an abelian Lie algebra,
is called the {\it abelianization} of $\mathfrak{g}$. 
\end{definition}

As the notation $H_1 (\mathfrak{g})$ indicates, 
there is a general theory of (co)homology of Lie algebras
due to Chevalley and Eilenberg, and
the above can be interpreted as the first homology
group of $\mathfrak{g}$.

Now suppose that the Lie algebra $\mathfrak{g}$ is graded. That is, 
there exists a direct sum decomposition 
\[\mathfrak{g} = \bigoplus_{i \ge 0}^\infty \mathfrak{g}(i)\]
such that $[\mathfrak{g}(k), \mathfrak{g}(l)] \subset \mathfrak{g}(k+l)$ for 
any $k, l \ge 0$. 
Then the homology group $H_\ast (\mathfrak{g})$ 
becomes bigraded. In particular, the abelianization is decomposed as 
\[H_1 (\mathfrak{g}) \cong \bigoplus_{k \ge 0} H_1 (\mathfrak{g})_k\]
where 
\[H_1 (\mathfrak{g})_k= \mbox{the quotient of $\mathfrak{g}(k)$ by 
$\displaystyle\sum_{\begin{subarray}{c}
i+j =k\\ i,j \ge 0
\end{subarray}} [\mathfrak{g}(i), \mathfrak{g}(j)]$}\]
is called the {\it weight} $k$ part of $H_1 (\mathfrak{g})$. 

If we set
$
\mathfrak{g}^+=\bigoplus_{i \ge 1}^\infty \mathfrak{g}(i)
\ \subset\
\mathfrak{g},
$
then it becomes an ideal of $\mathfrak{g}$ and we have an extension 
\begin{equation}
0\longrightarrow
\mathfrak{g}^+ \longrightarrow \mathfrak{g}
\longrightarrow \mathfrak{g}(0) \longrightarrow 0
\label{eq:ex}
\end{equation}
of Lie algebras, where the last map denotes the natural projection. It is easy to see that the above extension necessarily splits so that $\mathfrak{g}$ is isomorphic to the
semi-direct product $\mathfrak{g}^+\rtimes \mathfrak{g}(0)$.
The abelianization of $\mathfrak{g}^+$
can be described by
\[H_1 (\mathfrak{g}^+)_1 = \mathfrak{g}(1), \quad 
H_1 (\mathfrak{g}^+)_k= \mbox{the quotient of $\mathfrak{g}(k)$ by 
$\displaystyle\sum_{\begin{subarray}{c}
i+j =k\\ i,j \ge 1
\end{subarray}} [\mathfrak{g}(i), \mathfrak{g}(j)]$}\]
for $k \ge 2$. It follows that
the computation of $H_1 (\mathfrak{g}^+)$ 
is equivalent to 
the determination of a generating set of $\mathfrak{g}^+$ as a Lie algebra.

Finally, in the case of the graded Lie algebra over $\Q$,
the relation between the abelianizations of 
$\mathfrak{g}^+$ and $\mathfrak{g}$ is given by the
following Hochschild-Serre exact sequence
(see \cite{hs}, here we use the homology version rather
than the original cohomology version)
$$
H_2(\mathfrak{g})\longrightarrow
H_2(\mathfrak{g}(0))\longrightarrow
H_1(\mathfrak{g}^+)_{\mathfrak{g}(0)}\longrightarrow
H_1(\mathfrak{g})\longrightarrow
H_1(\mathfrak{g}(0))\longrightarrow
0.
$$
Here $H_1(\mathfrak{g}^+)_{\mathfrak{g}(0)}$ denotes the space
of coinvariants of $H_1(\mathfrak{g}^+)$ with respect to the
action of $\mathfrak{g}(0)$ on it. Since the extension \eqref{eq:ex} splits, the homomorphism
$H_{i}(\mathfrak{g})\rightarrow H_{i}(\mathfrak{g}(0))$
is {\it surjective} for any $i$ so that we have
a short exact sequence
$$
0 \longrightarrow
H_1(\mathfrak{g}^+)_{\mathfrak{g}(0)}\longrightarrow
H_1(\mathfrak{g})\longrightarrow
H_1(\mathfrak{g}(0))\longrightarrow
0
$$
which splits canonically.

%-----------------associative.tex--------2012/3/5, Sakasai ----
\section{Derivation Lie algebra of the free associative algebra}\label{sec:associative}

Let $H_n \cong \Z^n$ 
be a free abelian group of rank $n$ 
%$H_n \cong \Q^n$ an $n$-dimensional vector space over $\mathbb{Q}$ 
with a fixed ordered basis $\{ x_1, x_2,\ldots,x_n \}$. 
We suppose that $n \ge 2$. 
%The vector space $H_n$ can be also seen as 
%the first rational homology group of a free group of rank $n$. 
We write $H_n^\ast$ for the dual module $\Hom (H_n ,\Z)$. 
The dual basis of $H_n^\ast$ is denoted by 
$\{ x_1^\ast, x_2^\ast,\ldots,x_n^\ast \}$.

Let $T(H_n)= \displaystyle\bigoplus_{i=1}^\infty H_n^{\otimes i}$ 
denote the tensor algebra without constant terms generated by $H_n$. 
A {\it derivation} of $T(H_n)$ is an endomorphism $D$ of $T(H_n)$ 
satisfying
\begin{equation}\label{eq:derivation}
D(X \otimes Y) = D(X) \otimes Y + X \otimes D(Y)
\end{equation}
\noindent
for any $X,Y \in T(H_n)$. 
We denote the set of all derivations of $T(H_n)$ by $\Der (T(H_n))$, which 
has a natural structure of a module over $\Z$. Moreover we can endow 
$\Der (T(H_n))$ with a structure of a Lie algebra 
by restricting the bracket operation among endomorphisms of $T(H_n)$, 
namely 
\[
[F,G] = F \circ G - G \circ F
\]
for $F,G \in \Der (T(H_n))$. 

Note that a derivation is characterized by its action on 
the degree $1$ part $T(H_n)(1)=H_n$ as 
the definition (\ref{eq:derivation}) implies. Conversely, 
any homomorphism in $\Hom (H_n, T(H_n))$ defines a derivation of $T(H_n)$. 
Therefore 
we have a natural decomposition
\[\Der (T(H_n)) 
\cong \Hom (H_n, T(H_n)) \cong \bigoplus_{k \ge 0} 
\Der (T(H_n)) (k)\]
where
\[\Der (T(H_n)) (k) := \Hom(H_n, H_n^{\otimes (k+1)})=
H_n^\ast \otimes H_n^{\otimes (k+1)}\]
denotes the degree $k$ homogeneous part of $\Hom (H_n, T(H_n))$. 
Then for two elements 
\begin{align*}
F&=f \otimes u_1 \otimes u_2 \otimes \cdots \otimes u_{p+1} 
\in \Der (T(H_n))(p)=H_n^\ast \otimes H_n^{\otimes (p+1)}, \\
G&=g \otimes v_1 \otimes v_2 \otimes \cdots \otimes v_{q+1} 
\in \Der (T(H_n))(q)=H_n^\ast \otimes H_n^{\otimes (q+1)},
\end{align*}
\noindent
where $f, g \in H_n^\ast$ and $u_1,\ldots,u_{p+1}, 
v_1, \ldots, v_{q+1} \in H_n$, their bracket $[F,G] \in 
\Der (T(H_n))(p+q) =
H_n^\ast \otimes H_n^{\otimes (p+q+1)}$ is 
given by 
\begin{align}
[F,G] &= \sum_{s=1}^{q+1} f(v_s) \ 
g \otimes v_1 \otimes \cdots \otimes v_{s-1} \otimes 
(u_1 \otimes \cdots \otimes u_{p+1}) \otimes v_{s+1} 
\otimes \cdots \otimes v_{q+1} \label{eq:bracket}\\
&\quad -\sum_{t=1}^{p+1} g(u_t) \ 
f \otimes u_1 \otimes \cdots \otimes u_{t-1} \otimes 
(v_1 \otimes \cdots \otimes v_{q+1}) \otimes u_{t+1} 
\otimes \cdots \otimes u_{p+1}.\nonumber 
\end{align}
\noindent
%We can check that this bracket operation gives 
%a Lie algebra structure of $\Der (T(H_n))$. 
Note that 
$\Der (T(H_n))(0) = \Hom(H_n,H_n) \cong \gl{Z}$, 
where $\gl{Z}$ is the Lie algebra of all $(n \times n)$-matrices with entries in $\Z$. 

Let 
\[\Der^+ (T(H_n))=\displaystyle\bigoplus_{k \ge 1} \Der (T(H_n))(k)\]
be the Lie subalgebra of $\Der (T(H_n))$ 
consisting of all elements of positive degrees. 
We now compute $H_1 (\Der^+ (T(H_n)))$ in a stable range with respect to $n$. 

In \cite[Section 6]{morita_GT}, the first author introduced 
for $n \ge 2$ the homomorphism 
\[C_{13}:\Der (T(H_n))(2) \longrightarrow H_n^{\otimes 2}\]
defined by 
\[C_{13}(f \otimes u_1 \otimes u_2 \otimes u_3) = f(u_2) u_1 \otimes u_3,\]
where $f \in H_n^\ast$ and $u_1, u_2, u_3 \in H_n$
%is surjective and moreover that the composition 
and showed that the composition 
\[\wedge^2 \Der (T(H_n))(1) \xrightarrow{[\,\cdot\,,\,\cdot\,]} 
\Der (T(H_n))(2) \xrightarrow{C_{13}} H_n^{\otimes 2}\]
is trivial. Indeed, for $f, g \in H_n^\ast$ and $u_1,u_2,v_1,v_2 \in H_n$ 
we have 
\begin{align*}
[f \otimes u_1 \otimes u_2, g \otimes v_1 \otimes v_2]
&=f(v_1) g \otimes (u_1 \otimes u_2) \otimes v_2+
f(v_2) g \otimes v_1 \otimes (u_1 \otimes u_2) \\
& \quad -g(u_1) f \otimes (v_1 \otimes v_2) \otimes u_2 
-g(u_2) f \otimes u_1 \otimes (v_1 \otimes v_2) \\
& \stackrel{C_{13}}{\longmapsto}
f(v_1) g(u_2) u_1 \otimes v_2+
f(v_2) g(u_1) v_1 \otimes u_2 \\
& \qquad -g(u_1) f(v_2) v_1 \otimes u_2 
-g(u_2) f(v_1) u_1 \otimes v_2=0.
\end{align*}
\noindent
Since $C_{13}$ is clearly surjective, it induces an epimorphism  
$C_{13}:H_1 (\Der^+ (T(H_n)))_2 \twoheadrightarrow H_n^{\otimes 2}$. 

\begin{thm}\label{thm:associative}
$(1)$ For $n \ge 2$, we have a direct sum decomposition 
\[\Der (T(H_n))(2) = H_n^{\otimes 2} \oplus 
\big[\Der (T(H_n))(1),\Der (T(H_n))(1)\big].\]
In particular, the homomorphism $C_{13}: H_1(\Der^+ (T(H_n)))_2
\to H_n^{\otimes 2}$ 
is an isomorphism. 

\noindent
$(2)$ If $n \ge k \ge 3$, we have 
\begin{eqnarray*}
\hspace{5pt}\lefteqn{\Der (T(H_n))(k)}\hspace{15pt}\\
&=\big[\Der (T(H_n))(k-1),\Der (T(H_n))(1)\big] + 
\big[\Der (T(H_n))(k-2),\Der (T(H_n))(2)\big].
\end{eqnarray*}
\noindent
In particular, 
$H_1 (\Der^+ (T(H_n)))_k=0$ holds stably for any $k \ge 3$.
\end{thm}

\begin{remark}\label{rem:diagram}
The formula (\ref{eq:bracket}) for the bracket operation in 
$\Der (T(H_n))$ looks slightly complicated. However, by using the following 
diagrammatic description, we can make it clear and intuitive. 
Generators of 
$\Der (T(H_n))(k)=H_n^\ast \otimes H_n^{\otimes (k+1)}$ are written in the form 
\[x_l^\ast \otimes x_{i_1} \otimes x_{i_2} \otimes \cdots \otimes x_{i_{k+1}}\]
by using our basis. We associate to such a vector the diagram  as in 
Figure \ref{fig:ika}:

\begin{figure}[htbp]
\begin{center}
\includegraphics[width=0.25\textwidth]{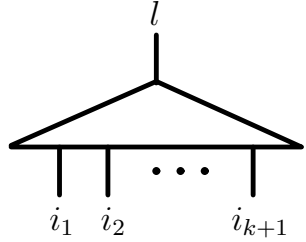}
\end{center}
\caption{The diagram for the vector 
$x_l^\ast \otimes x_{i_1} \otimes x_{i_2} \otimes \cdots \otimes x_{i_{k+1}}$}
\label{fig:ika}
\end{figure}

%\includegraphics[width=0.2\textwidth]{figure/ika_2_1.eps}
%\input{string3.tex}

%\vskip -87pt
%\hskip 60pt$l$

%\vskip 72pt
%\hskip 18pt $i_1$ \hskip 10pt $i_2$ \hskip 42pt $i_{k+1}$

\noindent
Then the formula is diagrammatically written 
as in Figure \ref{fig:bracket_An}, where we replace 
the diagrams in the right hand side under the rule 
shown in Figure \ref{fig:bracket_An_2}.

\begin{figure}[htbp]
\begin{center}
\includegraphics[width=0.55\textwidth]{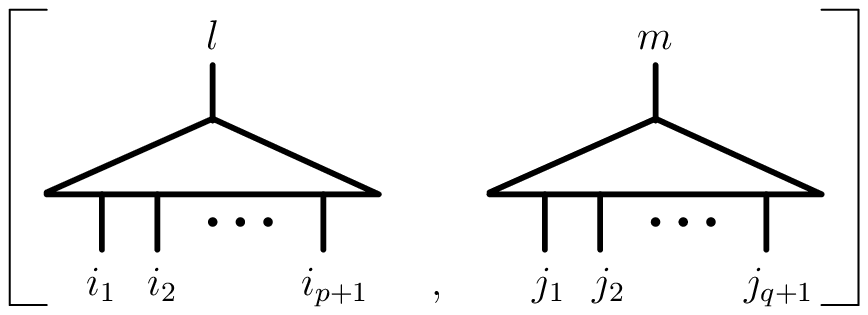}

\bigskip
\includegraphics[width=0.65\textwidth]{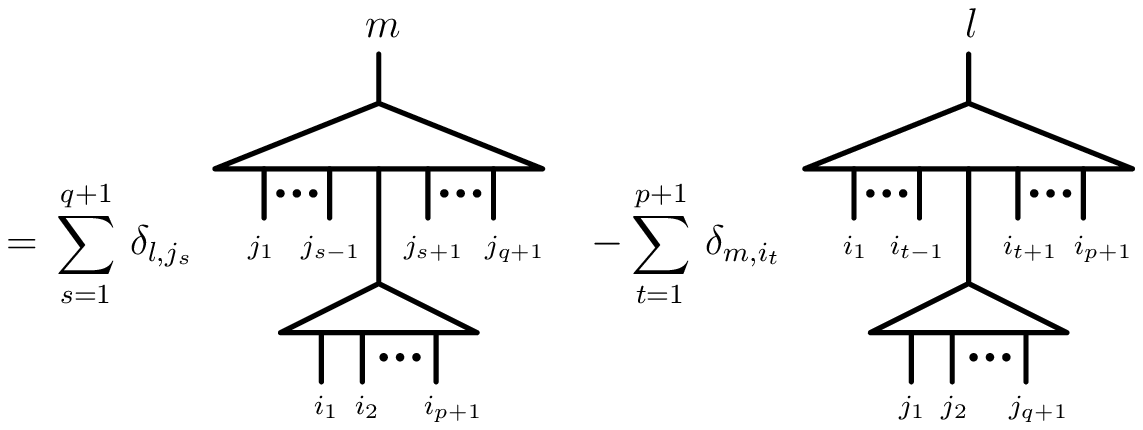}
\end{center}
\caption{Diagrammatic description of the bracket operation}
\label{fig:bracket_An}
\end{figure}

\begin{figure}[htbp]
\begin{center}
\includegraphics[width=0.55\textwidth]{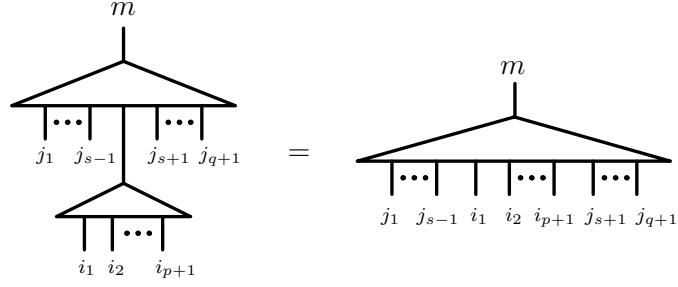}
\end{center}
\caption{Replace the diagram in Figure \ref{fig:bracket_An} 
(similarly for the second one)}
\label{fig:bracket_An_2}
\end{figure}
\end{remark}

\begin{proof}[Proof of Theorem $\ref{thm:associative}$]
%(1) We show that the image of the bracket map 
%\[ 
%[ \cdot , \cdot ] : \Der (T(H_n))(1) \otimes \Der (T(H_n))(1) 
%\longrightarrow \Der (T(H_n))(2) 
%\]
%is a module of rank $n^4 - n^2$. 
%
%First, the surjective homomorphism 
%$C_{13}:\Der (T(H_n))(2) \longrightarrow H_n^{\otimes 2}$ 
%has the kernel of rank $n^4 - n^2$ and 
%the kernel of $C_{13}$ contains the image of the bracket map. 
%Then we get that 
%the rank of $\Im [ \cdot , \cdot ]$ is less than or equal to the rank of $\ker C_{13} = n^4 - n^2$
%\[
%\dim \Im [ \cdot , \cdot ] \leq \dim \ker C_{13} = n^4 - n^2
%\]
%as vector spaces.  
%as modules. 

%On the other hand, 
(1) Define a section 
\[
s : H_n^{\otimes 2} \to \Der (T(H_n))(2)
\] 
of $C_{13}$ by 
$s(x_i \otimes x_j) = x_1^* \otimes x_i \otimes x_1 \otimes x_j$. 
Since $C_{13}\left(\big[\Der (T(H_n))(1),\Der (T(H_n))(1)\big]\right)=
\{0\}$ as already mentioned, 
%$C_{13}$ annihilates 
%$\big[\Der (T(H_n))(1),\Der (T(H_n))(1)\big]$, 
%$\Im [ \cdot , \cdot ]$, 
we have 
\[\big[\Der (T(H_n))(1),\Der (T(H_n))(1)\big]
\cap 
s (H_n^{\otimes 2})
%\Im [ \cdot , \cdot ] \cap \Im s 
= \{ 0 \}.
\]

The image of the bracket map contains the following types of elements. 

\begin{itemize}
\item 
$x_l^\ast \otimes x_{i_1} \otimes x_{i_2} \otimes x_{i_3} 
= [ x_l^\ast \otimes x_{i_1} \otimes x_{i_2}, x_l^\ast \otimes x_l \otimes x_{i_3}] 
\quad (l \neq i_1, i_2, i_3)$. 

\item 
$x_l^\ast \otimes x_l \otimes x_{i_1} \otimes x_{i_2} 
= [ x_{i_1}^\ast \otimes x_{i_1} \otimes x_{i_2}, x_l^\ast \otimes x_l \otimes x_{i_1}] 
\quad (l \neq i_1, i_2)$. 

\item 
$x_l^\ast \otimes x_{i_1} \otimes x_{i_2} \otimes x_l 
= [ x_{i_1}^\ast \otimes x_{i_1} \otimes x_{i_2}, x_l^\ast \otimes x_{i_1} \otimes x_l] 
\quad (l \neq i_1, i_2)$. 

\item 
$x_l^\ast \otimes x_l \otimes x_{i_1} \otimes x_l 
= [ x_l^\ast \otimes x_l \otimes x_{i_1}, x_l^\ast \otimes x_l \otimes x_l] 
\quad (l \neq i_1)$.
\end{itemize}
Moreover, we have for $l \ne 1$ 
\begin{itemize}
\item 
$x_{l}^\ast \otimes x_{i_1} \otimes x_{l} \otimes x_{i_2} = 
x_1^\ast \otimes x_{i_1} \otimes x_1 \otimes x_{i_2} - 
[ x_{l}^\ast \otimes x_1 \otimes x_{i_2}, x_1^\ast \otimes x_{i_1} \otimes x_{l}]$ 
\ \ $(l \neq i_1, i_2 \neq 1)$, \\
$x_{l}^\ast \otimes x_{i_1} \otimes x_{l} \otimes x_{i_2} = 
x_1^\ast \otimes x_{i_1} \otimes x_1 \otimes x_{i_2} - 
[ x_{l}^\ast \otimes x_{i_1} \otimes x_1, x_1^\ast \otimes x_{l} \otimes x_{i_2}]$ 
\ \ $(l \neq i_2, i_1 \neq 1)$, \\
$x_{l}^\ast \otimes x_{l} \otimes x_{l} \otimes x_{l} = 
x_1^\ast \otimes x_{l} \otimes x_1 \otimes x_{l} - 
[ x_{l}^\ast \otimes x_{l} \otimes x_1, x_1^\ast \otimes x_{l} \otimes x_{l}]$ \\
\qquad \qquad \qquad \qquad \qquad 
$+ [ x_{l}^\ast \otimes x_{l} \otimes x_{l}, x_1^\ast \otimes x_{l} \otimes x_1]$, \\
$x_{l}^\ast \otimes x_1 \otimes x_{l} \otimes x_1 = 
 x_1^\ast \otimes x_1 \otimes x_1 \otimes x_1 - 
[ x_1^\ast \otimes x_1 \otimes x_1, x_{l}^\ast \otimes x_1 \otimes x_{l}]$ \\ 
\qquad \qquad \qquad \qquad \qquad 
$+ [ x_1^\ast \otimes x_1 \otimes x_{l}, x_{l}^\ast \otimes x_1 \otimes x_1]$. 
\end{itemize}

Since the above elements and $s (H_n^{\otimes 2})$ 
generate $\Der (T(H_n))(2)$, 
the claim (1) holds.  

%\begin{align*}
%& x_1^\ast \otimes x_{i_2} \otimes x_1 \otimes x_{i_3} - 
%x_{i_1}^\ast \otimes x_{i_2} \otimes x_{i_1} \otimes x_{i_3} \\
%&= 
%\quad (i_1 \neq i_2, i_3 \neq 1), \\
%& x_1^\ast \otimes x_{i_2} \otimes x_1 \otimes x_{i_3} - 
%x_{i_1}^\ast \otimes x_{i_2} \otimes x_{i_1} \otimes x_{i_3} \\
%&= 
%[ x_{i_1}^\ast \otimes x_{i_2} \otimes x_1, x_1^\ast \otimes x_{i_1} \otimes x_{i_3}] 
%\quad (i_1 \neq i_3, i_2 \neq 1), \\

%&x_1^\ast \otimes x_{i_1} \otimes x_1 \otimes x_{i_1} - 
%x_{i_1}^\ast \otimes x_{i_1} \otimes x_{i_1} \otimes x_{i_1} \\
%&= 
%[ x_{i_1}^\ast \otimes x_{i_1} \otimes x_1, x_1^\ast \otimes x_{i_1} \otimes x_{i_1}] 
%- [ x_{i_1}^\ast \otimes x_{i_1} \otimes x_{i_1}, x_1^\ast \otimes x_{i_1} \otimes x_1], \\
%& x_1^\ast \otimes x_1 \otimes x_1 \otimes x_1 - 
%x_{i_1}^\ast \otimes x_1 \otimes x_{i_1} \otimes x_1 \\
%&= 
%[ x_1^\ast \otimes x_1 \otimes x_1, x_{i_1}^\ast \otimes x_1 \otimes x_{i_1}] 
%- [ x_1^\ast \otimes x_1 \otimes x_{i_1}, x_{i_1}^\ast \otimes x_1 \otimes x_1]. 
%\end{align*}

%The above elements are linearly independent. 
%Furthermore, the number of their elements is 
%\[
%n (n-1)^3 + n (n-1)^2 + n (n-1)^2 + n (n-1) + n^2 (n-1) = n^4 - n^2 . 
%\]
%Therefore we obtain that the rank of $\Im [ \cdot , \cdot ]$ is $n^4 - n^2$. 

(2) We now exhibit an algorithm to rewrite 
a generator 
$x_l^\ast \otimes x_{i_1} \otimes x_{i_2} \otimes \cdots \otimes x_{i_{k+1}}$ 
of $\Der (T(H_n))(k)$ as an element in 
$\big[\Der (T(H_n))(k-1),\Der (T(H_n))(1)\big] + 
\big[\Der (T(H_n))(k-2),\Der (T(H_n))(2)\big]$. 

(Case 1) When $l \neq i_1, i_2, \ldots, i_{k+1}$, we have an 
equality 
\[
x_l^\ast \otimes x_{i_1} \otimes x_{i_2} \otimes \cdots \otimes x_{i_{k+1}}
=[x_l^\ast \otimes x_{i_k} \otimes x_{i_{k+1}}, \ 
x_l^\ast \otimes x_{i_1} \otimes x_{i_2} \otimes \cdots \otimes x_{i_{k-1}} 
\otimes x_l]
\]
as depicted in Figure \ref{fig:geso12} 
and we have done. 
\begin{figure}[htbp]
\begin{center}
\begin{picture}(420,100)
\put(0,20){\includegraphics[scale=0.6]{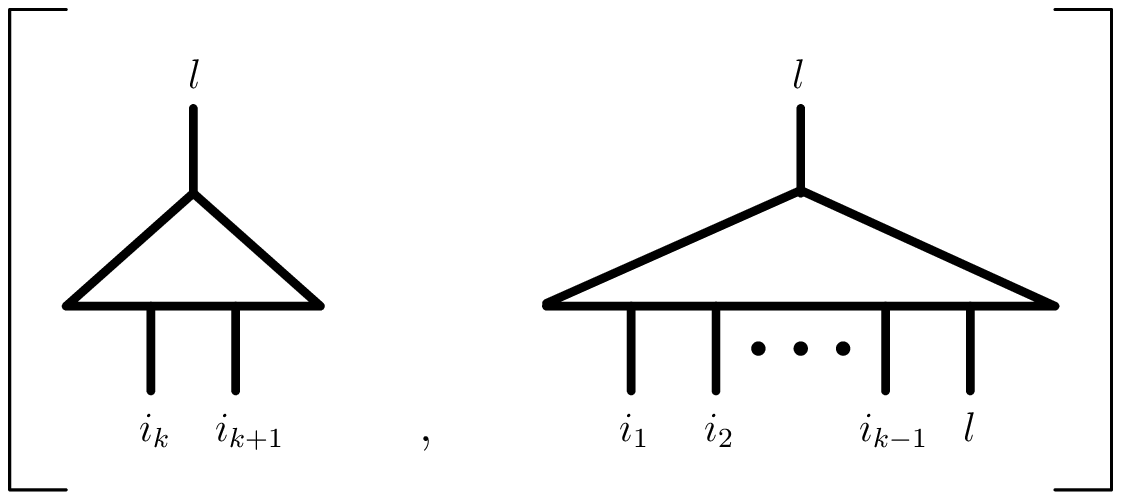}}
\put(200,0){\includegraphics[scale=0.6]{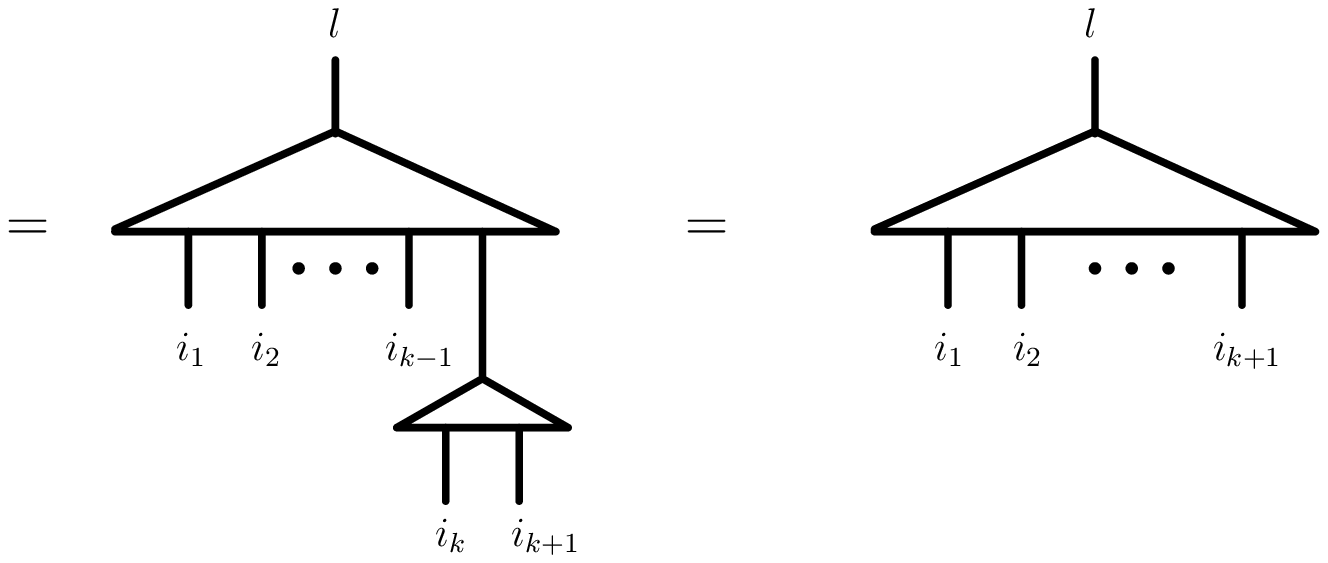}}
\end{picture}
\end{center}
\caption{Case $1$ of Proof of Theorem \ref{thm:associative}(2)}
\label{fig:geso12}
\end{figure}

(Case 2) Suppose that $l$ coincides with 
only one of $i_1, i_2,\ldots,i_{k+1}$ (say $l=i_j$). 
We rename $\{i_1, \ldots, i_{j-1}, i_{j+1},\ldots, i_{k+1}\}$ 
by $\{j_1, j_2, \ldots, j_k\}$ so that $j_p \neq l$ for $1 \le p \le k$. 
By assumption, we have $k \ge 3$. 

The equality 
\[[x_{j_1}^\ast \otimes x_{j_1} \otimes x_{j_2} \otimes \cdots \otimes 
x_{j_k}, \ 
x_l^\ast \otimes x_l \otimes x_{j_1} ]
=x_l^\ast \otimes x_l \otimes x_{j_1} \otimes x_{j_2} 
\otimes \cdots \otimes x_{j_k}\]
%\begin{figure}[htbp]
%\begin{center}
%\includegraphics[width=0.6\textwidth]{geso3_labeled.eps}
%\end{center}
%\caption{}
%\label{fig:geso3}
%\end{figure}
shows that the right hand side is in 
$\big[\Der (T(H_n))(k-1),\Der (T(H_n))(1)\big]$. 
%$\Im [\,\cdot\,,\,\cdot\,]$. 
Now we ``slide'' $x_l$ to any other slot as follows. 

If $q-p \ge 2$, then 
$x_l^\ast \otimes x_{j_{p+1}} \otimes x_{j_{p+2}} \otimes 
\cdots \otimes x_{j_q} \in \Der^+(T(H_n))$ and 
we have 
\begin{align*}
&[x_l^\ast \otimes x_{j_{p+1}} \otimes x_{j_{p+2}} \otimes 
\cdots \otimes x_{j_q}, \ 
x_l^\ast \otimes x_{j_1} \otimes \cdots \otimes 
x_{j_p} \otimes x_l \otimes x_l \otimes x_{j_{q+1}} \otimes 
\cdots \otimes x_{j_k}]\\
= \ & x_l^\ast \otimes x_{j_1} \otimes \cdots x_{j_q} 
\otimes x_l \otimes x_{j_{q+1}} \otimes \cdots \otimes x_{j_k}
+
x_l^\ast \otimes x_{j_1} \otimes \cdots \otimes 
x_{j_p} \otimes x_l \otimes x_{j_{p+1}} \otimes 
\cdots \otimes x_{j_k},
\end{align*}
which implies that modulo brackets and up to sign, 
we can slide $x_l$ to the right by at least two slots, 
as depicted in Figure \ref{fig:geso45}. 
\begin{figure}[htbp]
\begin{center}
\begin{picture}(420,90)
\put(0,10){\includegraphics[scale=0.6]{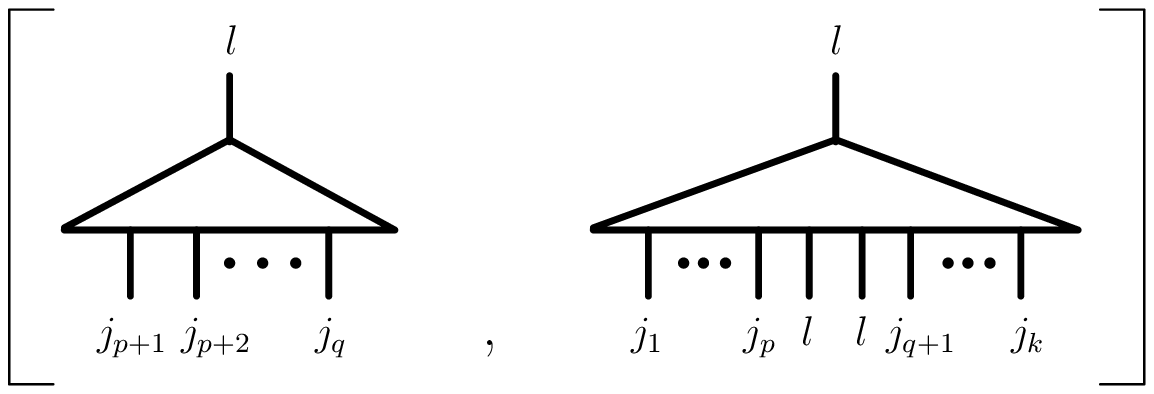}}
\put(205,0){\includegraphics[scale=0.6]{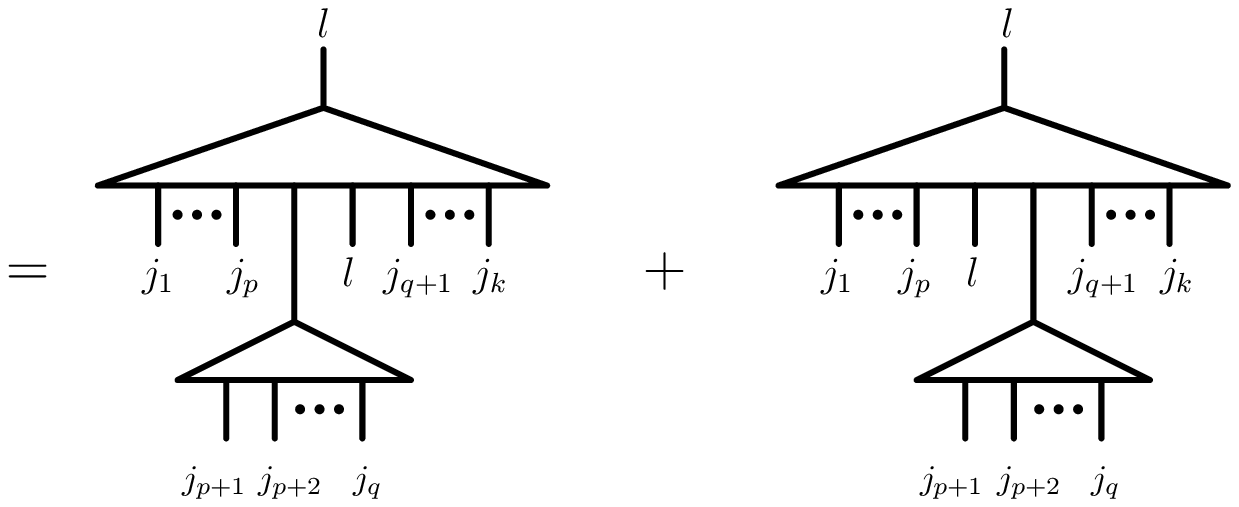}}
\end{picture}
\end{center}
\caption{Slide $x_l$}
\label{fig:geso45}
\end{figure}

\noindent
By applying this observation to 
$x_l^\ast \otimes x_l \otimes x_{j_1} \otimes x_{j_2} 
\otimes \cdots \otimes x_{j_k} \in 
\big[\Der (T(H_n))(k-1),\Der (T(H_n))(1)\big]$, we see that 
\begin{align*}
x_l^\ast \otimes x_{j_1} \otimes x_l \otimes 
x_{j_2} \otimes x_{j_3} 
\otimes \cdots \otimes x_{j_k}&\equiv 
x_l^\ast \otimes x_{j_1} \otimes x_{j_2} \otimes x_{j_3} 
\otimes x_l \otimes x_{j_4} \otimes \cdots \otimes x_{j_k}\\
&\equiv 
x_l^\ast \otimes x_l \otimes x_{j_1} \otimes x_{j_2} 
\otimes \cdots \otimes x_{j_k} \\
&\equiv 0
\end{align*}
\noindent
modulo $\big[\Der (T(H_n))(k-1),\Der (T(H_n))(1)\big] +
\big[\Der (T(H_n))(k-2),\Der (T(H_n))(2)\big]$. 
Starting from 
$x_l^\ast \otimes x_l \otimes x_{j_1} \otimes x_{j_2} 
\otimes \cdots \otimes x_{j_k}$ and 
$x_l^\ast \otimes x_{j_1} \otimes x_l \otimes 
x_{j_2} \otimes x_{j_3} 
\otimes \cdots \otimes x_{j_k}$, we can slide $x_l$ to 
any other slot modulo $\big[\Der (T(H_n))(k-2),\Der (T(H_n))(2)\big]$. 
Hence Case 2 is done.

(Case 3) Here we consider the general case. 
For $x_l^\ast \otimes x_{i_1} \otimes x_{i_2} \otimes \cdots 
\otimes x_{i_{k+1}}$, we take 
\[m \in \{1,2,\ldots,n \} - \{i_1, i_2, \ldots, i_{k-1}\}\]
where $\{1,2,\ldots,n \} - \{i_1, i_2, \ldots, i_{k-1}\}\neq \emptyset$ by 
the assumption that $n \ge k$. Then we 
have 
\begin{align*}
x_l^\ast \otimes x_{i_1} \otimes x_{i_2} \otimes \cdots 
\otimes x_{i_{k+1}} &=
[x_m^\ast \otimes x_{i_k} \otimes x_{i_{k+1}}, \ 
x_l^\ast \otimes x_{i_1} \otimes x_{i_2} \otimes \cdots 
\otimes x_{i_{k-1}} \otimes x_m]\\
&\quad +\delta_{l, i_k} 
x_m^\ast \otimes x_{i_1} \otimes x_{i_2} \otimes \cdots 
\otimes x_{i_{k-1}} \otimes x_m \otimes x_{i_{k+1}}\\
&\quad +\delta_{l, i_{k+1}} 
x_m^\ast \otimes x_{i_k} \otimes x_{i_1} \otimes x_{i_2} \otimes \cdots 
\otimes x_{i_{k-1}} \otimes x_m 
\end{align*}
as depicted in Figure \ref{fig:geso67}. 
\begin{figure}[htbp]
\begin{center}
\begin{picture}(410,100)
\put(0,10){\includegraphics[scale=0.6]{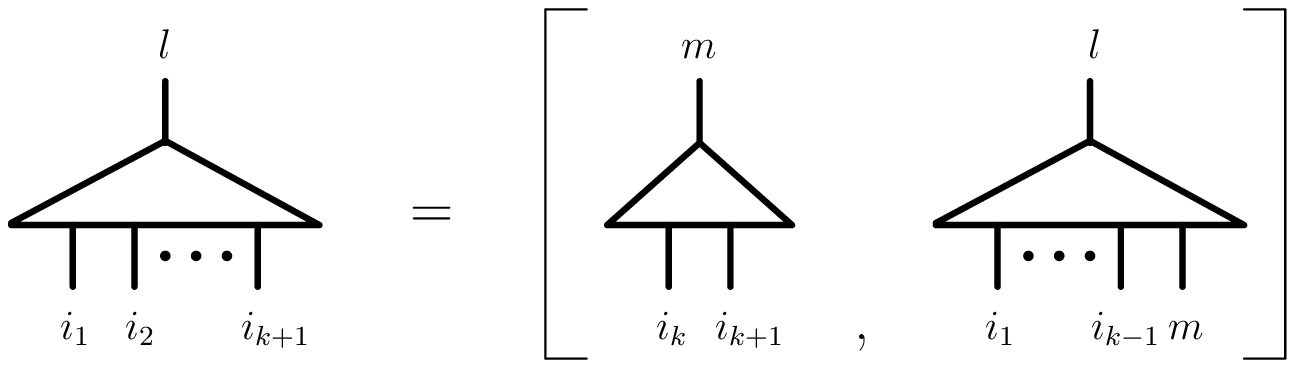}}
\put(230,0){\includegraphics[scale=0.6]{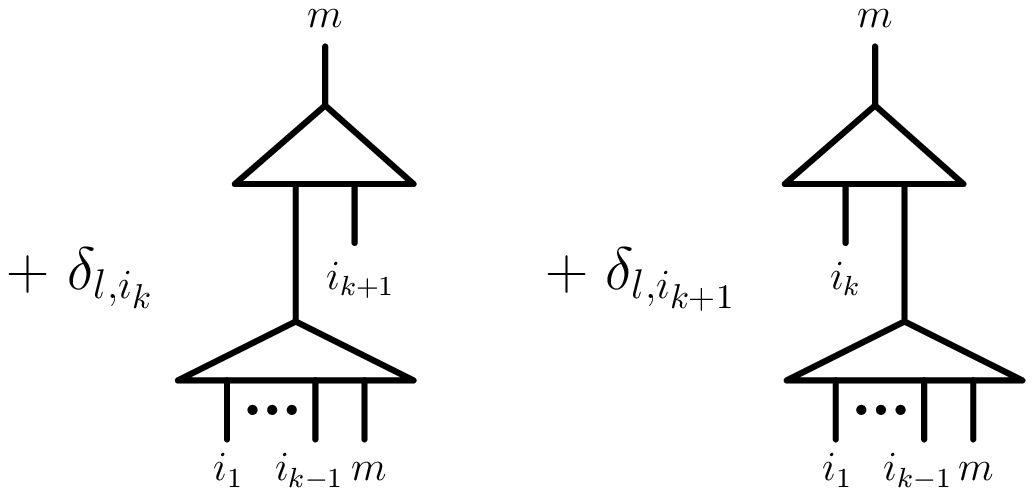}}
\end{picture}
\end{center}
\caption{Case $3$ of Proof of Theorem \ref{thm:associative}(2)}
\label{fig:geso67}
\end{figure}

\noindent
If $m \neq i_{k+1}$, the second term of the right hand side is reduced to Case 2. 
Otherwise, we consider the equality
\begin{align*}
&[x_{i_1}^\ast \otimes x_{i_1} \otimes x_{i_2} \cdots \otimes x_{i_{k-1}}, \ 
x_m^\ast \otimes x_{i_1} \otimes x_m \otimes x_m ]\\
=\ &
x_m^\ast \otimes x_{i_1} \otimes x_{i_2} \otimes \cdots 
\otimes x_{i_{k-1}} \otimes x_m \otimes x_m.
\end{align*}
Then this term belongs to $\big[\Der (T(H_n))(k-2),\Der (T(H_n))(2)\big]$. 
Similarly, 
if $m \neq i_k$, the third term has already been considered in Case 2. 
In the other case $m = i_k$, we have 
\begin{align*}
&[x_{i_1}^\ast \otimes x_{i_1} \otimes x_{i_2} \cdots \otimes x_{i_{k-1}}, \ 
x_m^\ast \otimes x_m \otimes x_{i_1} \otimes x_m ]\\
=\ &
x_m^\ast \otimes x_m \otimes x_{i_1} \otimes x_{i_2} \otimes \cdots 
\otimes x_{i_{k-1}} \otimes x_m.
\end{align*}
%in 
%$\big[\Der (T(H_n))(k-2),\Der (T(H_n))(2)\big]$ 
%by Case 2.
%Hence it remains to show that the second term is in 
%$\big[\Der (T(H_n))(k-1),\Der (T(H_n))(1)\big] \oplus 
%\big[\Der (T(H_n))(k-2),\Der (T(H_n))(2)\big]$. 
%As for the second term, if $m \neq i_{k+1}$, it reduces to Case 2. Otherwise, take 
%\[m' \in \{1,2,\ldots,n\}-\{i_1, i_2, \ldots, i_{k-1},m\} 
%\neq \emptyset\]
%and consider the equality
%\begin{align*}
%&[x_{m'}^\ast \otimes x_{i_1} \otimes x_{i_2}, \ 
%x_m^\ast \otimes x_{m'} \otimes x_{i_3} \otimes \cdots 
%\otimes x_{i_{k-1}} \otimes x_m \otimes x_m]\\
%=\ &
%x_m^\ast \otimes x_{i_1} \otimes x_{i_2} \otimes \cdots 
%\otimes x_{i_{k-1}} \otimes x_m \otimes x_m.
%\end{align*}
Therefore this term also belongs to $\big[\Der (T(H_n))(k-2),\Der (T(H_n))(2)\big]$. 
This completes the proof. 
\end{proof}

Now we prove Theorem \ref{thm:h1} (i). 
More precisely we show the following. 
\begin{cor}\label{cor:associative}
$(1)$ For any $n \ge 2$, The natural pairing $H_n^\ast \otimes H_n \to \Z$ 
induces an isomorphism 
$H_1 (\Der (T(H_n)))_0 \cong \Z$. 

\noindent
$(2)$ For any $n \ge 2$, we have 
$H_1 (\Der (T(H_n)))_1=0$. 

\noindent
$(3)$ If $n \ge k \ge 2$, we have $H_1 (\Der (T(H_n)))_k=0$. 
\end{cor}
\begin{proof}
(1) The above pairing corresponds to the usual trace map 
\[
H_n^* \otimes H_n = \mathfrak{gl}(n , {\mathbb Z})  
\longrightarrow {\mathbb Z}.
\]
It can be easily checked that any traceless matrix is 
in $\Im [ \cdot, \cdot ]$ over ${\mathbb Z}$. 

(2) If we apply the argument in Remark \ref{rmk:B} below 
to the case $k=1$, we can conclude that any element of degree $1$ 
is contained in $\Im [ \cdot , \cdot ]$. 

(3) By Theorem \ref{thm:associative} (1), 
it suffices to show that the composition 
\[
\Der (T(H_n))(0) \otimes \Der (T(H_n))(2) 
\xrightarrow{[\cdot,\cdot]}
\Der (T(H_n))(2) 
\xrightarrow{C_{13}}
H_n^{\otimes 2}
\]
is surjective. 
For $i,j \in \{ 1,2, \ldots,n \}$ with $i \ne j$, we have 
\begin{align*}
C_{13} ( [ x_j^* \otimes x_i, x_j^* \otimes x_j \otimes x_j \otimes x_i ]) 
&= x_i \otimes  x_i , \\
C_{13} ( [ x_j^* \otimes x_j, x_i^* \otimes x_i \otimes x_i \otimes x_j ]) 
&= x_i \otimes  x_j .
\end{align*}
This completes the proof.

%By applying the argument in Section \ref{sec:homology} to 
%the split extension 
%\[0 \longrightarrow \Der^+ (T(H_n)) 
%\longrightarrow \Der (T(H_n))
%\longrightarrow \Der(T(H_n))(0) \longrightarrow 0\]
%of Lie algebras, 
%we obtain an exact sequence 
%\[0 \longrightarrow H_1(\Der^+ (T(H_n)))_{\mathfrak{gl}} 
%\longrightarrow H_1 (\Der (T(H_n)))
%\longrightarrow H_1(\Der(T(H_n))(0)) \longrightarrow 0.\]
%Here we have $H_1 (\Der(T(H_n))(0)) = H_1 (\gl{Z})$ as mentioned above 
%and it is easy to see that 
%the natural pairing $H_n^\ast \otimes H_n \to \Z$ gives an 
%isomorphism $H_1 (\gl{Z}) \cong \Z$. 
%By Theorem \ref{thm:associative}, we have 
%\[H_1(\Der^+ (T(H_n)))_{\mathfrak{gl}} = 
%((H_n^* \otimes H_n^{\otimes 2}) \oplus H_n^{\otimes 2})_{\mathfrak{gl}}
%=0\]
%in the range $n \ge k \ge 2$. This completes the proof.
\end{proof}

\begin{remark}\label{rmk:B}
One of the referees kindly points out that, over the rationals, 
the abelianization of $\Der (T(H_n) \otimes {\mathbb Q})$ 
can be determined easily as 
\[
H_1 (\Der (T(H_n)) \otimes {\mathbb Q}) \cong \Q 
\]
for {\it all} $n \geq 2$ by the following argument. 
%In fact, the 
The identity map $I$ belongs to 
$\Der (T(H_n) \otimes \Q) (0) \cong \mathfrak{gl}(n,\Q)$
and for any $D \in \Der (T(H_n) \otimes \Q) (k)$ \, ($k \geq 1$), 
we have 
\[
[I , D] = I \circ D - D \circ I = (k + 1) D - D = k D. 
\]
%A similar statement holds for $\Der (\Ln) \otimes \Q$ 
%treated in the next section. 
He or she also points out that the same argument can be applied to the case of 
$\Der(\Ln) \otimes \Q$, treated in the next section, as well.

%If we consider an $n$-dimensional vector space over $\Q$ as $H_n$, 
%then Theorem \ref{thm:associative} can be shown by argument of representation theory, 
%where the derivation Lie algebra $\Der (T (H_n))$ is 
%regarded as a Lie algebra over $\Q$.
%The argument in this section works also 
%for corresponding modules over $\mathbb{Z}$. That is, 
%we may start from putting $H_n \cong \Z^n$. 
%Then the derivation Lie algebra $\Der (T (H_n))$ 
%is given as a Lie algebra over $\Z$ and 
%the same conclusion 
%(after replacing vector spaces with modules and 
%linear maps with homomorphisms) is obtained. 
\end{remark}

%-----------------lie.tex--------2012/3/2, MSS----
\section{Derivation Lie algebra of the free Lie algebra}\label{sec:lie}

Let $\Ln$ denote the free Lie algebra generated by $H_n$.
This Lie algebra is naturally graded and we have a direct sum decomposition 
$\Ln= \displaystyle\bigoplus_{i=1}^\infty \Ln(i)$.  For small degree $i$, 
the module $\Ln(i)$ is given by 
\[\Ln(1)=H_n, \quad \Ln(2) \cong \wedge^2 H_n, 
\quad \Ln(3) \cong (H_n \otimes (\wedge^2 H_n))/\wedge^3 H_n, \quad 
\ldots\]
where $\wedge^2 H_n$ and $\wedge^3 H_n$ correspond to 
the anti-symmetry and the Jacobi identity of the bracket operation of $\Ln$.

A {\it derivation} of $\Ln$ is an endomorphism $D$ of $\Ln$ 
satisfying
\[D([X,Y]) = [D(X), Y] + [X, D(Y)]\]
for any $X,Y \in \Ln$. 
We denote by $\Der(\Ln)$ the set of all derivations of $\Ln$. 
By an argument similar to the case of $\Der(T(H_n))$, 
we have a natural decomposition 
\[\Der (\Ln) 
\cong \Hom (H_n, \Ln) \cong \bigoplus_{k \ge 0} \Der(\Ln)(k)\]
where
\[\Der(\Ln)(k):= \Hom(H_n, \Ln(k+1))=
H_n^\ast \otimes \Ln(k+1)\]
denotes the degree $k$ homogeneous part of $\Hom (H_n, \Ln)$. 
Also, we can endow 
$\Der(\Ln)$ with a graded Lie algebra structure 
by restricting the bracket operation among endomorphisms of $\Der (\Ln)$. 
%with respect to the above 
%decomposition by a formula, which is mentioned later by using diagrams, 
%similar to $\Der(T(H_n))$. 
Again we have $\Der(\Ln)(0) = \Hom(H_n,H_n) \cong \gl{Z}$. 

It is easily checked that for each $k \ge 1$ the module $\Ln(k+1)$ is 
generated by elements of the form
\[[x_{i_1},x_{i_2},\ldots,x_{x_{i_{k+1}}}]:=
[[\cdots[[x_{i_1},x_{i_2}],x_{i_3}],\ldots],x_{i_{k+1}}].\]
Therefore $\Der(\Ln)(k)$ is generated by elements of the form
\[x_l^\ast \otimes [x_{i_1},x_{i_2},\ldots,x_{x_{i_{k+1}}}].\]

\begin{remark}\label{rem:diagram_Lie}
As in the case of $\Der (T(H_n))$, 
the following diagrammatic description for 
$\Der(\Ln)$ is helpful and should be well-known.
The module $\Ln$ is generated by 
rooted binary planar trees, each of whose trivalent vertices 
has a cyclic order and each of whose univalent vertices other than the root 
is colored by an integer in 
$\{1,2,\ldots,n\}$ 
corresponding to the basis of $H_n$, modulo anti-symmetry and IHX relations. 
For example, the element 
$[[[x_{i_1},x_{i_2}],x_{i_3}],[x_{i_4},x_{i_5}]] \in \Ln (5)$ is assigned to 
the left diagram of Figure \ref{fig:tree}. We can extend this description to 
a diagrammatic description for 
$\Der(\Ln)$ by labeling the root by an integer corresponding 
to the basis of $H_n^\ast$. The right diagram of Figure \ref{fig:tree} 
represents 
$x_l^\ast \otimes [[[x_{i_1},x_{i_2}],x_{i_3}],[x_{i_4},x_{i_5}]] \in 
\Der(\Ln(4))$. 

\begin{figure}[htbp]
\begin{center}
\includegraphics[width=0.6\textwidth]{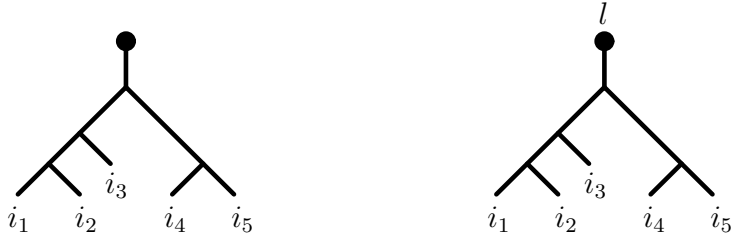}
\end{center}
\caption{The diagrams for 
$[[[x_{i_1},x_{i_2}],x_{i_3}],[x_{i_4},x_{i_5}]] \in \Ln (5)$ (left) and 
$x_l^\ast \otimes [[[x_{i_1},x_{i_2}],x_{i_3}],[x_{i_4},x_{i_5}]] \in 
\Der(\Ln)(4)$ (right)}
\label{fig:tree}
\end{figure}

\noindent
The bracket operation for generators is diagrammatically given as 
in Figure \ref{fig:bracket_Ln}.

\begin{figure}[htbp]
\begin{center}
\includegraphics[width=0.50\textwidth]{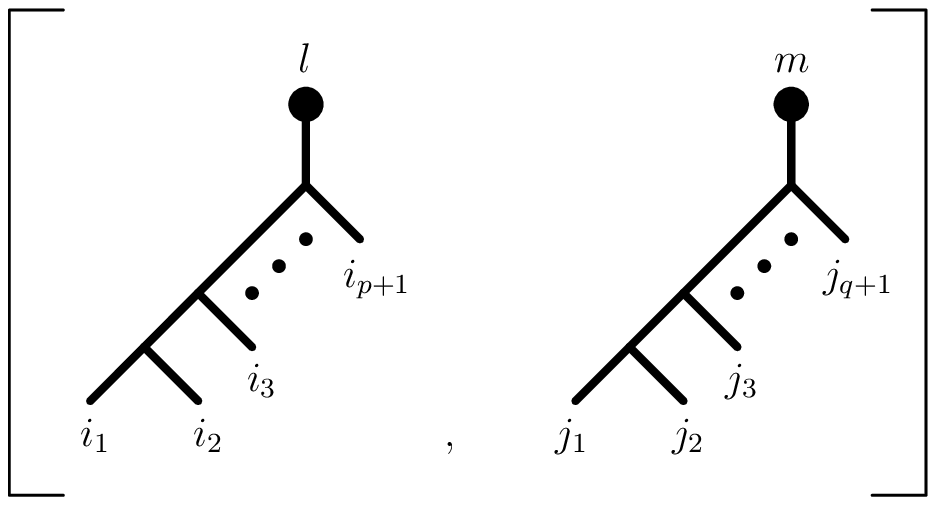}

\bigskip
\includegraphics[width=0.78\textwidth]{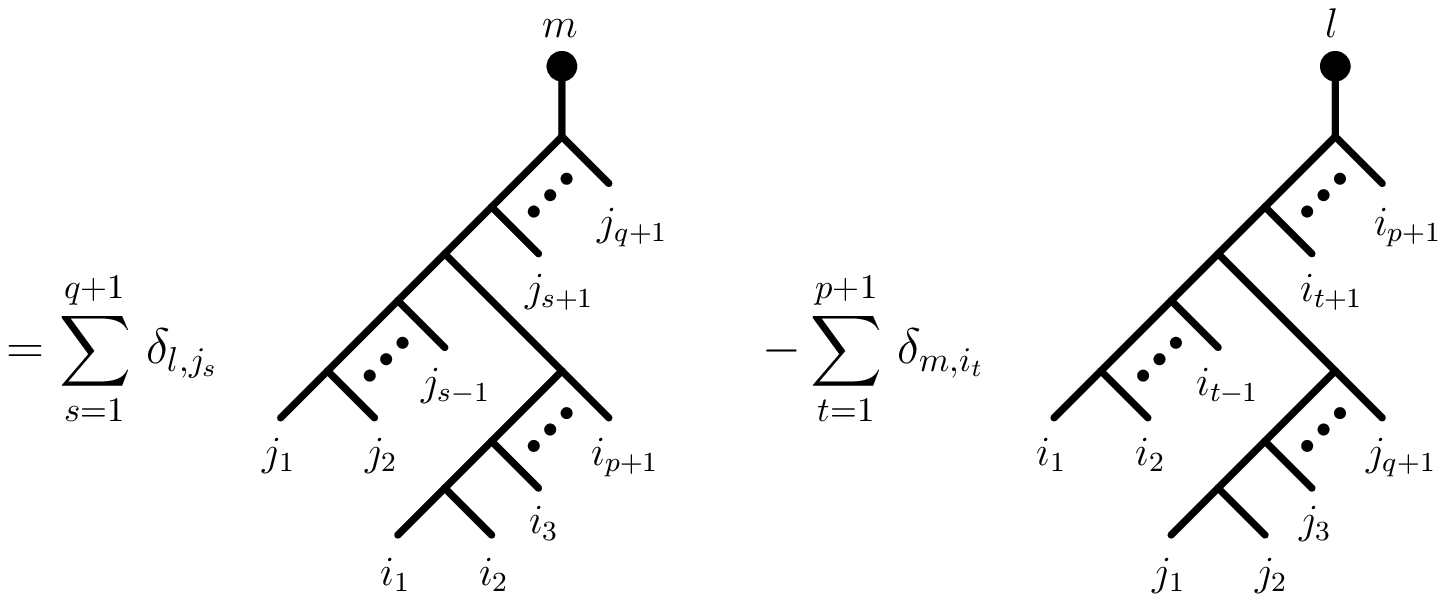}
\end{center}
\caption{Diagrammatic description of the bracket operation}
\label{fig:bracket_Ln}
\end{figure}
\end{remark}

Let $\Der^+ (\Ln)=\displaystyle\bigoplus_{k \ge 1} \Der(\Ln) (k)$ be 
the Lie subalgebra of $\Der(\Ln)$ consisting of all elements of positive 
degrees. 
The abelianization $H_1 (\Der^+(\Ln) \otimes {\mathbb Q})$, 
over the rationals rather than the integers, 
in a certain stable range was 
first computed by Kassabov \cite[Theorem 1.4.11]{kassabov}. 
To explain the result, we recall the {\it trace map} introduced by 
the first author \cite{morita93}. 

It is well known that the Lie algebra $\Ln$ can be embedded in 
$T(H_n)$ by replacing the bracket $[X,Y]$ with $X \otimes Y-Y \otimes X$ 
repeatedly. This operation keeps the degree. 
Then consider a sequence of homomorphisms 
\[\Der (\Ln)(k) = H_n^\ast \otimes 
\Ln(k+1) \longrightarrow H_n^\ast \otimes H_n^{\otimes (k+1)} 
\xrightarrow{C_{12}} H_n^{\otimes k} \longrightarrow S^k H_n,\]
where the first map is the above mentioned embedding, the map $C_{12}$ takes 
the pairing of $H_n^\ast$ and the first component of $H_n^{\otimes (k+1)}$ and 
the last map is the symmetrization map to the $k$-th symmetric power 
of $H_n$. We put the composition by $tr_k$, namely
\[tr_k:\Der (\Ln)(k) \longrightarrow S^k H_n.\]
It was shown in \cite{morita93} that $tr_k$ vanishes on 
$\Im [\,\cdot\,,\,\cdot\,]$. 
Kassabov's theorem \cite[Theorem 1.4.11]{kassabov}, 
reformulated by the trace maps, says 
that $tr_k: H_1(\Der^+ (\Ln) \otimes {\mathbb Q})_k \to S^k H_n \otimes {\mathbb Q}$ 
is an isomorphism if $n(n-1) \ge k \ge 2$. 
%Now we show the following. 
Now we show the following result which gives a slight improvement of the 
above theorem of Kassabov.

\begin{thm}\label{thm:lie}
If $n \ge k+2 \ge 4$, we have a direct sum decomposition 
\[\Der(\Ln) (k)=S^k H_n \oplus 
\big[\Der(\Ln)(k-1), \Der(\Ln)(1)\big]\]
where the first projection is given by the trace map $tr_k$. 
In particular, the induced map 
$tr_k: H_1(\Der^+ (\Ln))_k \to S^k H_n$ 
gives an isomorphism stably for any $k \ge 2$.
\end{thm}

\begin{remark}\label{rem:kassabov}
%While we can derive the abelianization of $\Der^+ (\Ln) \otimes {\mathbb Q}$ 
%from the above theorem, our stable range is {\it weaker} than Kassabov's one. 
%However, our statement has the following advantages. 
%
%The assumption that $n \ge k+2 \ge 4$ makes our statement for 
%the computation of $H_1(\Der^+ (\Ln))$ {\it weaker} 
%than that of Kassabov \cite{kassabov}, where 
%he showed (after reformulated by the trace maps) 
%
%However, our statement has the following advantages. 
%
%at the cost of the weaker stable range. 
Our proof of the above theorem is very close to the original argument of
Kassabov. Although our stable range is {\it weaker} than his one,
our statement has the following advantages.
\begin{itemize}
\item \ The proof works over ${\mathbb Z}$. %(especially by using diagrams). 
\item \ We show that 
$\big[\Der(\Ln)(k-1), \Der(\Ln)(1)\big] = 
\displaystyle\sum_{\begin{subarray}{c}
i+j =k \\ i,j \ge 1
\end{subarray}} \big[\Der(\Ln)(i), \Der(\Ln)(j)\big]$, 
namely, any element of $\Im [ \cdot, \cdot ]$ of degree $k$ can be expressed as 
a linear combination of the brackets of elements of degree $1$ and $k-1$. 
%which says 
%more than the computation of the abelianization. 
%\item \ The argument works also for corresponding modules over $\mathbb{Z}$.
\end{itemize}
\end{remark}
\begin{proof}[Proof of Theorem $\ref{thm:lie}$]
Let $\Gamma$ be the generating set of $\Der(\Ln) (k)$ consisting of 
all elements of the form
\[x_l^\ast \otimes [x_{i_1},x_{i_2}, \ldots, x_{i_{k+1}}].\]
First we make the set $\Gamma$ smaller 
as a generating set of the quotient 
\[Q:=\Der(\Ln) (k)/ \big[\Der(\Ln)(k-1), \Der(\Ln)(1)\big].\]

Suppose an element 
$x_l^\ast \otimes [x_{i_1},x_{i_2}, \ldots, x_{i_{k+1}}] \in \Gamma$ is 
given. 

(Case 1) When $l \neq i_1, i_2$, we have an equality
\begin{align*}
x_l^\ast \otimes [x_{i_1},x_{i_2}, \ldots, x_{i_{k+1}}]
&=[x_l^\ast \otimes [x_{i_1},x_{i_2}], \ 
x_l^\ast \otimes [x_l,x_{i_3}, x_{i_4}, \ldots, x_{i_{k+1}}]]\\
&\quad -\sum_{j=3}^{k+1} 
\delta_{i_j,l} 
x_l^\ast \otimes [x_l,x_{i_3}, \ldots, x_{i_{j-1}}, 
[x_{i_1},x_{i_2}],x_{i_{j+1}}, \ldots, x_{i_{k+1}}].
\end{align*}
\noindent
The second term of the right hand side is rewritten as 
\begin{align*}
-\sum_{j=3}^{k+1} 
\delta_{i_j,l} & (
x_l^\ast \otimes [x_l,x_{i_3}, \ldots, x_{i_{j-1}}, 
x_{i_1},x_{i_2},x_{i_{j+1}}, \ldots, x_{i_{k+1}}]\\
&-x_l^\ast \otimes [x_l,x_{i_3}, \ldots, x_{i_{j-1}}, 
x_{i_2},x_{i_1},x_{i_{j+1}}, \ldots, x_{i_{k+1}}])
\end{align*}
\noindent
by applying the Jacobi identity
\[[X,[x_{i_1},x_{i_2}]]=
-[x_{i_1},[x_{i_2},X]]-[x_{i_2},[X,x_{i_1}]]
=-[[X,x_{i_2}],x_{i_1}]+[[X,x_{i_1}],x_{i_2}]\]
with $X=[x_l,x_{i_3}, \ldots, x_{i_{j-1}}]$. Therefore 
the quotient $Q$ can be generated by the elements 
$x_l^\ast \otimes [x_{i_1},x_{i_2}, \ldots, x_{i_{k+1}}]$ 
in $\Gamma$ with $l=i_1$ and $l \neq i_2$. 

(Case 2) For an element 
$x_l^\ast \otimes [x_l,x_{i_2}, \ldots, x_{i_{k+1}}]$ with 
$l \neq i_2$, we take an integer $m$ from 
the set $\{1,2,\ldots,n\}-\{i_2,i_3,\ldots, i_{k+1}\}$ which is 
not empty. Then we have 
\begin{align*}
x_l^\ast \otimes [x_l,x_{i_2}, \ldots, x_{i_{k+1}}]
&=[x_m^\ast \otimes [x_l,x_{i_2}], \ 
x_l^\ast \otimes [x_m,x_{i_3}, x_{i_4}, \ldots, x_{i_{k+1}}]]\\
&\quad + x_m^\ast \otimes [x_m,x_{i_3}, x_{i_4}, \ldots, x_{i_{k+1}}, 
x_{i_2}].
\end{align*}
\noindent
This shows that 
the quotient $Q$ can be generated by the elements in $\Gamma$ of the form 
$x_l^\ast \otimes [x_l,x_{i_2}, \ldots, x_{i_{k+1}}]$ 
with $l \neq i_2, i_3,\ldots,i_{k+1}$. 

(Case 3) Suppose an element 
$x_l^\ast \otimes [x_l,x_{i_2}, \ldots, x_{i_{k+1}}]$ 
of $\Gamma$ with $l \neq i_2, i_3,\ldots,i_{k+1}$ is given. 
For every integer $j$ with $2 \le j \le k$, we apply the Jacobi identity 
to $[[Y,x_{i_j}],x_{i_{j+1}}]$ with 
$Y=[x_l,x_{i_2},\ldots, x_{i_{j-1}}]$. Then we have an equality 
\begin{align*}
x_l^\ast \otimes [x_l,x_{i_2}, \ldots, x_{i_j},x_{i_{j+1}},
\ldots,x_{i_{k+1}}] &=
x_l^\ast \otimes [x_l,x_{i_2}, \ldots, x_{i_{j-1}},
[x_{i_j},x_{i_{j+1}}], x_{i_{j+2}}, \ldots, x_{i_{k+1}}]\\
& \quad +x_l^\ast \otimes [x_l,x_{i_2}, \ldots, x_{i_{j-1}},
x_{i_{j+1}},x_{i_j}, x_{i_{j+2}}, \ldots, x_{i_{k+1}}].
\end{align*}
As for the first term of the right hand side, we 
take an integer $m$ 
from $\{1,2,\ldots,n\}-\{l,i_2,i_3,\ldots, i_{k+1}\}$ which is 
not empty and consider the equality 
\begin{align*}
&x_l^\ast \otimes [x_l,x_{i_2}, \ldots, x_{i_{j-1}},
[x_{i_j},x_{i_{j+1}}], x_{i_{j+2}}, \ldots, x_{i_{k+1}}]\\
&=
[x_m^\ast \otimes [x_{i_j},x_{i_{j+1}}],\ 
x_l^\ast \otimes [x_l,x_{i_2}, \ldots, x_{i_{j-1}},
x_m, x_{i_{j+2}}, \ldots, x_{i_{k+1}}]].
\end{align*}
\noindent
Consequently the equality 
\[x_l^\ast \otimes [x_l,x_{i_2}, \ldots, x_{i_j},x_{i_{j+1}},\ldots, x_{i_{k+1}}]=x_l^\ast \otimes [x_l,x_{i_2}, \ldots, x_{i_{j-1}},
x_{i_{j+1}},x_{i_j}, x_{i_{j+2}}, \ldots, x_{i_{k+1}}]\]
holds as an element of the quotient $Q$. In particular, we see that 
the element $x_l^\ast \otimes [x_l,x_{i_2}, \ldots, x_{i_{k+1}}]$ in $Q$ with 
$l \neq i_2, i_3,\ldots,i_{k+1}$ 
is invariant under the permutation of the indices 
$x_{i_2}, x_{i_3},\ldots,x_{i_{k+2}}$. Moreover the equality 
\begin{align*}
x_l^\ast \otimes [x_l,x_{i_2}, \ldots, x_{i_{k+1}}]&=
[x_m^\ast \otimes [x_l,x_{i_2}],\ x_l^\ast \otimes 
[x_m,x_{i_3}, \ldots, x_{i_{k+1}}]]\\
&\quad +
x_m^\ast \otimes [x_m,x_{i_3}, \ldots, x_{i_{k+1}}, x_{i_2}]
\end{align*}
\noindent
shows that as elements of the quotient $Q$, we have 
\[x_l^\ast \otimes [x_l,x_{i_2}, \ldots, x_{i_{k+1}}]=
x_m^\ast \otimes [x_m,x_{i_3}, \ldots, x_{i_{k+1}}, x_{i_2}]=
x_m^\ast \otimes [x_m,x_{i_2}, \ldots, x_{i_{k+1}}]\]
as long as $l,m \neq i_2, i_3,\ldots,i_{k+1}$. 

For every $k \ge 2$, define a homomorphism $\Phi_k:S^k H_n \to Q$ by 
\[\Phi_k (x_{i_2}x_{i_3}\cdots x_{i_{k+1}})=
x_l^\ast \otimes [x_l,x_{i_2}, \ldots, x_{i_{k+1}}]\]
where $l$ is chosen for each generator $x_{i_2}x_{i_3} \cdots x_{i_{k+1}}$ 
of $S^k H_n$ so that $l \neq i_2, i_3,\ldots,i_{k+1}$. 
The argument in the previous paragraphs shows that $\Phi_k$ is 
well-defined, independent of the choices of $l$ and surjective. 

On the other hand, 
since $tr_k$ vanishes on $\Im [ \cdot , \cdot ]$ as already mentioned, 
we have a homomorphism 
\[tr_k:Q \longrightarrow S^k H_n\]
and it is easily checked that 
\[tr_k(x_l^\ast \otimes [x_l,x_{i_2}, \ldots, x_{i_{k+1}}]) 
=x_{i_2}x_{i_3}\cdots x_{i_{k+1}}\]
if $l \neq i_2, i_3,\ldots,i_{k+1}$. Therefore 
we have $tr_k \circ \Phi_k= \mathrm{id}_{S^k H_n}$ implying that 
$tr_k:Q \to S^k H_n$ is an isomorphism. 
This completes the proof. 
\end{proof}

Now we prove Theorem \ref{thm:h1} (ii). 
More precisely we show the following. 
\begin{cor}
$(1)$ For any $n \ge 2$, 
the natural pairing $H_n^\ast \otimes H_n \to {\mathbb Z}$ 
induces an isomorphism 
$H_1 (\Der (\Ln))_0 \cong {\mathbb Z}$.

\noindent
$(2)$ For any $n \ge 2$, we have $H_1 (\Der (\Ln))_1=0$. 

\noindent
$(3)$ If $n \ge k+2 \ge 4$, we have $H_1 (\Der (\Ln))_k=0$. 
\end{cor}

\begin{proof}
By an argument similar to the proof of Corollary \ref{cor:associative}, 
(1) and (2) follow immediately. 
(3) follows from 
Theorem \ref{thm:lie} 
%the above computations 
%together with 
and the equality 
\[
x_l^* \otimes [x_l,x_{i_2} , \ldots,x_{i_{k+1}}] 
= 
[x_l^* \otimes [x_l,x_{i_2} , \ldots,x_{i_k},x_m], x_m^* \otimes x_{i_{k+1}}],
\]
where $m \in \{1,2,\ldots,n \} - \{l, i_2, \ldots, i_{k+1}\} \ne \emptyset$. 
%Because this shows that 
Indeed they show that 
any element is in $\Im [ \cdot, \cdot]$ 
if we are allowed to use elements of degree $0$. 
\end{proof}
\begin{remark}
As was mentioned in Remark \ref{rmk:B}, 
one of the referees points out that the 
abelianization of $\Der(\Ln)\otimes \Q$ can be
easily determined as $H_1(\Der(\Ln) \otimes\Q) \cong\Q$ 
because the identity map I belongs to $\Der(\Ln)(0)\otimes\Q$.
\end{remark}

%-----------------sympassociative.tex--------2013/1/5, MSS--
\section{Symplectic derivation Lie algebra of the free associative algebra}\label{sec:symp}

Let $\Sg$ be a closed connected oriented 
surface of genus $g \ge 2$. 
The first integral homology group $H_1 (\Sg)$ 
of $\Sg$ is isomorphic to 
a free abelian group $H_{2g}$ of rank $2g$. 
%, namely we have 
%$H_1 (\Sg) \cong H_{2g}$. 
This 
module has a natural intersection form 
\[\mu : H_1 (\Sg) \otimes H_1 (\Sg) \longrightarrow \Z\]
which is non-degenerate and skew-symmetric. 
Let $\{ a_1,\ldots,a_g,b_1,\ldots,b_g \}$ be 
a symplectic basis of $H_1 (\Sg)$ 
with respect to $\mu$, namely 
\[\mu(a_i,a_j)=0,\quad \mu(b_i,b_j)=0,\quad
\mu(a_i,b_j)=\delta_{ij}.\]
The Poincar\'e duality gives a canonical isomorphism between 
$H_1 (\Sg)$ and its 
dual module $H_1 (\Sg)^\ast =H^1 (\Sg)$, the 
first integral cohomology group of $\Sg$. In 
this isomorphism, $a_i$ (resp.\ $b_i$) $\in H_1 (\Sg)$ 
corresponds to $b_i^\ast$ 
(resp.\ $-a_i^\ast$) $\in H^1 (\Sg)$ 
where $\{ a_1^\ast ,\ldots,a_g^\ast ,b_1^\ast ,\ldots,b_g^\ast \}$ 
is the dual basis of $H^1 (\Sg)$. 
We denote these canonically 
isomorphic modules by $H$ for simplicity. 
We write $\mathrm{Sp}(H)$ for the symplectic transformation 
group of $H$. It consists of all automorphisms of $H$ 
preserving $\mu$. 

Denote the symplectic class by 
\[\omega_0 = \displaystyle\sum_{i=1}^g (a_i \otimes b_i - b_i \otimes a_i) 
\in H \otimes H,\]
which is independent of the choice of a symplectic basis of $H$ and 
is invariant under the action of $\mathrm{Sp}(H)$. 
A derivation $D \in \Der (T(H)) \cong \Der (T(H_{2g}))$ 
is said to be {\it symplectic} 
if it satisfies 
$D(\omega_0)=0$. It is easily checked that 
the set of all symplectic derivations 
forms a Lie subalgebra of $\Der (T(H))$. 

In this section, we shall consider the rational forms of the above modules. 
Put $H_\Q=H \otimes \Q$ and define a derivation of $H_\Q$ to be 
a linear map from $T(H_\Q)$ to itself satisfying the 
same formula as in Section \ref{sec:associative}.
Then we have 
$\Der (T(H_\Q)) \cong 
%\Der (T(H) \otimes \Q) \cong 
\Der (T(H))\otimes \Q$ as Lie algebras over $\Q$ and it is naturally graded. 
Let $\ag$ be the subspace of $\Der (T(H_\Q))$ consisting of all 
symplectic derivations. 
It is a Lie subalgebra of $\Der (T(H_\Q))$. 
This Lie algebra was first studied by Kontsevich 
\cite{kontsevich1,kontsevich2} (see Section \ref{sec:moduli}). 
A grading of $\ag$ is induced from $\Der (T(H_\Q))$ and 
define $\ag (k)$ to be its degree $k$ homogeneous part. We have a direct 
sum decomposition 
\[\ag = \bigoplus_{k \ge 0} \ag(k).\]
We also define a Lie subalgebra 
$\ag^+:=\displaystyle\bigoplus_{k \ge 1} \ag(k)$ consisting 
of all derivations of positive degrees. 
Note that the symplectic transformation 
group $\mathrm{Sp}(H_\Q)$ of $H_\Q$ acts on $\ag (k)$ for each $k$. 

%We also define a Lie subalgebra 
%$\ag^+:=\displaystyle\bigoplus_{k \ge 1} \ag(k)$ consisting 
%of all derivations of positive degrees. 
Using the identification 
\[\Hom (H_\Q, H_\Q^{\otimes (k+1)})= H_\Q^\ast 
\otimes H_\Q^{\otimes (k+1)}=
H_\Q^{\otimes (k+2)},\] 
we can rewrite 
the symplecticity of a derivation of $H_\Q$ as follows 
(see also \cite[Proposition 2]{morita_GT}). By definition, a symplectic 
derivation $D$ satisfies that 
\begin{equation}\label{eq:sympderiv}
0=D (\omega_0)=\sum_{i=1}^g
\big(D(a_i) \otimes b_i+a_i \otimes D(b_i)
-D(b_i)\otimes a_i -b_i \otimes D(a_i) \big).
\end{equation}
\noindent
Since $D \in \Hom (H_\Q, H_\Q^{\otimes (k+1)})$ corresponds to 
\[\sum_{i=1}^g \big(a_i^\ast \otimes D(a_i) + b_i^\ast \otimes D(b_i)\big) 
=\sum_{i=1}^g \big(-b_i \otimes D(a_i) + a_i \otimes D(b_i)\big) =:D^\ast\]
in $H_\Q^\ast \otimes H_\Q^{\otimes (k+1)}=
H_\Q^{\otimes (k+2)}$, the above equality (\ref{eq:sympderiv}) says that 
\[D^\ast =\sigma_{k+2} (D^\ast),\]
where $\sigma_{k+2}$ is a generator of the 
cyclic group $\mathbb{Z}/(k+2)\mathbb{Z}$ acting on $H_\Q^{\otimes (k+2)}$ 
by 
\[\sigma_{k+2} (u_1 \otimes u_2 \otimes \cdots \otimes u_{k+2})=
u_2 \otimes \cdots \otimes u_{k+2} \otimes u_1.\] 
Consequently, the degree $k$ 
part $\ag(k) \subset \Hom (H_\Q, H_\Q^{\otimes (k+1)})$ 
is rewritten as 
\[\ag (k) = \left( H_\Q^{\otimes (k+2)}\right)^{\mathbb{Z}/(k+2)\mathbb{Z}},\]
where the right hand side is the invariant part of 
$H_\Q^{\otimes (k+2)}$ with respect to the action of the 
group $\mathbb{Z}/(k+2)\mathbb{Z}$. 
%as cyclic permutations of elements in 
%$H^{\otimes (k+2)}$. In particular, 
From this description, we can see 
that $\ag (0) \cong \mathfrak{sp}(H_\Q) \cong S^2 H_\Q$, 
the symplectic Lie algebra, and that  
\[\ag(1)=\left( H_\Q^{\otimes 3}\right)^{\mathbb{Z}/3\mathbb{Z}} \cong 
S^3 H_\Q \oplus \wedge^3 H_\Q.\]
%where $\wedge^i H$ denotes the $i$-th exterior power of $H$. 

Now we focus on the abelianizations of $\ag$ and $\ag^+$. 
First we consider the latter. The weight $1$ part $H_1 (\ag^+)_1$ of 
$H_1 (\ag^+)$ is given by 
$\ag(1)$. 
The weight $2$ part $H_1 (\ag^+)_2$ was calculated by 
the first author in \cite[Theorem 6]{morita_GT} and it is given by 
\[H_1 (\ag^+)_2 \cong \wedge^2 H_\Q/\langle \omega_0 \rangle\]
as $\mathrm{Sp}(H_\Q)$-modules, where $\langle \omega_0 \rangle$ 
denotes the submodule of $\wedge^2 H_\Q$ spanned by 
$\omega_0$ as an element of $\wedge^2 H_\Q \subset H_\Q \otimes H_\Q$. 
In fact, an argument similar to the one just before Theorem \ref{thm:associative} 
shows that the composition 
\[\ag (2) \hookrightarrow H_\Q^\ast \otimes H_\Q^{\otimes 3} 
\xrightarrow{C_{13}} H_\Q \otimes H_\Q \xrightarrow{\text{proj.}} 
\wedge^2 H_\Q/\langle \omega_0 \rangle\]
is an $\mathrm{Sp}(H_\Q)$-equivariant epimorphism which 
annihilates $[\ag (1), \ag (1)]$. 
Then a direct calculation shows that this map just gives $H_1 (\ag^+)_2$. 

The main result of this section is the following: 
\begin{thm}\label{thm:symp}
If $g \ge k+3 \ge 6$, then $H_1 (\ag^+)_k=0$. 
\end{thm}
\noindent
For the proof of this theorem, we use more diagrammatic-minded argument 
than those in the previous cases. We introduce 
{\it spiders} and {\it chord diagrams} which play 
important roles in our proof. 

The vector space $\ag (k) = \left( 
H_\Q^{\otimes (k+2)}\right)^{\mathbb{Z}/(k+2)\mathbb{Z}}$ 
is generated by vectors of the form 
\[S(i_1,i_2,\ldots,i_{k+2}):=
\sum_{j=1}^{k+2} \sigma_{k+2}^j (a_{i_1} \otimes a_{i_2} \otimes 
\cdots \otimes a_{i_{k+2}}),\]
where $i_1, i_2 ,\ldots, i_{k+2} \in \{\pm 1,\pm 2,\ldots,\pm g\}$ and 
$a_l:=b_{-l}$ for $l<0$. 
% and $\sigma_{k+2}$ is a generator of $\mathbb{Z}/(k+2)\mathbb{Z}$. 
We call such a vector $S(i_1,i_2,\ldots,i_{k+2})$ a {\it spider\/} 
(see also Conant-Vogtmann \cite{cv}). 
In a natural way, we can represent a spider in $\ag (k)$  
by a graph with one 
$(k+2)$-valent vertex and $(k+2)$ univalent vertices, each of which 
is colored by an element in $\{\pm 1,\pm 2,\ldots,\pm g\}$ corresponding 
to the symplectic basis of $H_\Q$ and is connected by an edge called a {\it leg\/} 
to the $(k+2)$-valent 
vertex. The edges (and hence vertices) are ordered cyclically. 
For example, the left of Figure \ref{fig:spider_chord} represents 
the spider $S(1,4,-2,-1,3,-1,2,1)=S(4,-2,-1,3,-1,2,1,1)=\cdots$.

For two spiders $S_1=S(i_1,i_2,\ldots,i_{p+2}) \in \ag (p)$ and 
$S_2=S(j_1,j_2,\ldots,j_{q+2}) \in \ag (q)$, 
their bracket $[S_1,S_2] \in \ag (p+q)$ 
is diagrammatically given 
by the formula shown in Figure \ref{fig:spider_bracket}.

\begin{figure}[htbp]
\begin{center}
\includegraphics[width=0.95\textwidth]{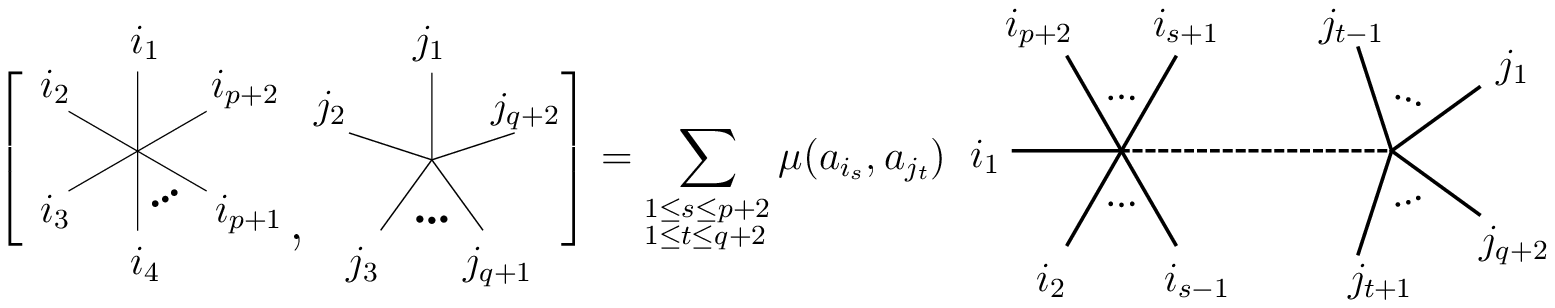}
\end{center}
\caption{Bracket of spiders, where the dashed line in the right hand 
side is collapsed to a point to make a new spider}
\label{fig:spider_bracket}
\end{figure}

To a spider $S$, we associate a {\it chord diagram} $C(S)$ 
(in a generalized sense) so that 
the vertices of $C(S)$ are ordered cyclically and colored 
according to the legs of $S$ and 
two vertices are connected by a chord if their colors differ by sign. 
(Two vertices with the same color are not connected.) 
We identify a spider with the corresponding chord diagram. 

\begin{figure}[htbp]
\begin{center}
\includegraphics[width=0.6\textwidth]{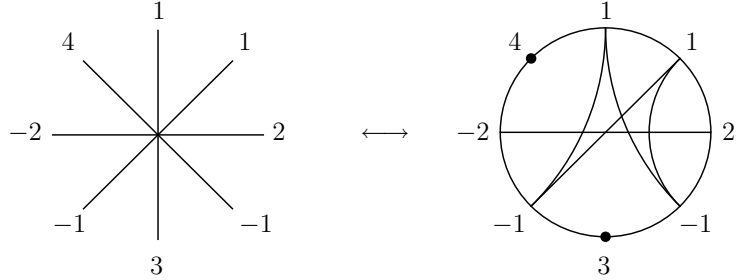}
\end{center}
\caption{A spider and a chord diagram}
\label{fig:spider_chord}
\end{figure}

\begin{definition}\label{def:classification}
A vertex $v$ of a chord diagram is said to be 
\begin{itemize}
\item[(a)] {\it unpaired} if it is not connected to any other vertex 
by a chord. 

\item[(b)] {\it single paired} if it is connected to only one other vertex, 
say $w$, by a chord and $w$ is connected to only $v$. 

\item[(c)] {\it multiple paired} if it is neither unpaired nor single paired. 
\end{itemize}
\end{definition}
\noindent
By abuse of notation, we also say ``a color $i$ is {\it unpaired}\/'', 
``a chord is {\it single paired\/}'', etc. 

\begin{definition}\label{def:multiplicity}
For a chord diagram $C$, its {\it multiplicity} $m(C)$ is defined by 
\[m(C)=2 \mbox{(number of chords)}-\mbox{(number of vertices having chords)}.\]
\end{definition}
\noindent
For example, the multiplicity of the chord diagram in Figure 
\ref{fig:spider_chord} is $4$. 
The multiplicity of a chord diagram without multiple paired vertices 
is zero by definition. Note that the multiplicity only depends on the set of colors of the diagram.

\begin{definition}\label{def:separable}
A chord diagram $C$ is said to be {\it separable} if there exists an 
arc inside the outer circle of $C$ 
connecting two points of the outer circle which are not vertices of $C$ 
such that each region separated by the arc has at least two vertices and 
the arc does not intersect with the chords. 
\end{definition}
\begin{lem}\label{lem:separable}
If $g \ge k+3 \ge 6$ and 
the chord diagram $C(S)$ of a spider $S$ is separable, then 
$C(S)$ is in $\Im [\,\cdot\,,\,\cdot\,]$.
\end{lem}
\begin{proof}
Cut the chord diagram $C(S)$ by an arc separating it and 
for each region glue the two endpoints of the piece of the outer circle. 
We put vertices for the identified points and give them 
colors with opposite sign that are distinct from those possessed by $C(S)$, 
which is possible by the assumption $g \ge k+3$. 
The new chord diagrams $C_1$ and $C_2$ satisfy 
$[C_1,C_2]=C(S)$ (see Figure \ref{fig:separable} as an example). 
\end{proof}

\begin{figure}[htbp]
\begin{center}
\includegraphics[width=0.7\textwidth]{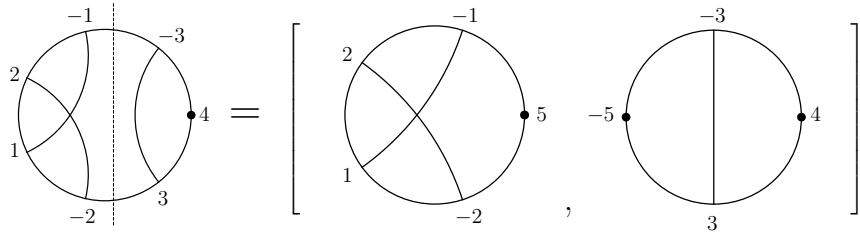}
\end{center}
\caption{A separable chord diagram}
\label{fig:separable}
\end{figure}

In the proof of Theorem \ref{thm:symp}, the following specific form of chord diagrams 
plays a key role. 
\begin{definition}\label{def:standard_form}
A chord diagram $C$ is said to be of the {\it standard form} if 
it corresponds to one of the following spiders
\begin{itemize}
\item $S(c_1, c_2, -c_1, c_3, -c_2, c_4, -c_3, c_5, -c_4, c_6, -c_5, \ldots, c_m,
-c_{m-1}, d_1, 
-c_m, d_2)$, 
\item $S(c_1, c_2, -c_1, c_3, -c_2, c_4, -c_3, c_5, -c_4, c_6, -c_5, \ldots, c_m,
-c_{m-1}, d_1, 
-c_m)$,
\item $S(c_1, c_2, -c_1, c_3, -c_2, c_4, -c_3, c_5, -c_4, c_6, -c_5, \ldots, c_m,
-c_{m-1},-c_m, d_2)$,
\item $S(c_1, c_2, -c_1, c_3, -c_2, c_4, -c_3, c_5, -c_4, c_6, -c_5, \ldots, c_m,
-c_{m-1},-c_m)$,
\end{itemize}
\noindent
where the colors $\pm c_1, \pm c_2, \ldots, \pm c_{m}, d_1, d_2$ 
(some of $c_i,d_1, d_2$ might be negative) 
are mutually distinct. In particular, 
$C$ does not have multiple paired vertices.  
Diagrammatically, a chord diagram of the standard form is given as in 
Figure \ref{fig:standard}. 
\end{definition}

\begin{figure}[htbp]
\begin{center}
\includegraphics[width=0.3\textwidth]{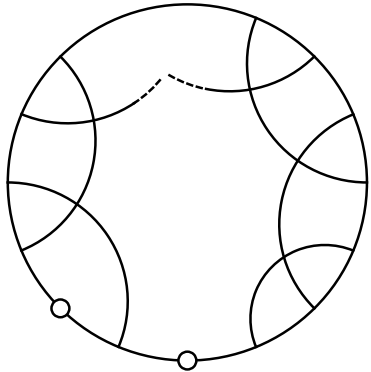}
\end{center}
\caption{The standard form, where {\it each of white vertices might not exist\/}}
\label{fig:standard}
\end{figure}

\begin{lem}\label{lem:standard_form}
If $g \ge k+3 \ge 6$, the quotient 
$\ag(k)/\displaystyle\sum_{\begin{subarray}{c}
i+j =k \\ i,j \ge 1
\end{subarray}} [\ag(i),\ag(j)]$ is 
generated by spiders corresponding to chord diagrams of the standard form. 
%which we call the {\it standard form}, shown in Figure $\ref{fig:standard}$.
\end{lem}
\begin{proof}
It suffices to exhibit an algorithm by which 
a given chord diagram $C(S)$ corresponding to a spider $S$ 
is rewritten modulo brackets as a linear combination of chord diagrams of 
the standard form. 

Suppose we are given a chord diagram $C$ 
corresponding to a spider with multiplicity $m(C)$. 
We may assume that $C$ is not separable. 

If $C$ does not have a single paired chord, we take 
two adjacent vertices. 
By using the colors $i$, $j$ of these vertices, we can write 
$C=C(S(i,j,X))$ for some word $X$ of colors with length bigger 
than $2$. Then we have 
\begin{align*}
S(i,j,X) &= [S(X,n),\ S(-n,i,j)]\\
&\quad +\sum_{\tiny\begin{subarray}{c}
\mbox{color $-i$}\\
\mbox{in $X$}
\end{subarray}} \pm S(n,Z_1,j,-n,Z_2) 
+\sum_{\tiny\begin{subarray}{c}
\mbox{color $-j$}\\
\mbox{in $X$}
\end{subarray}} \pm S(n,Z_1,-n,i,Z_2) 
\end{align*}
where $n>0$ and $-n$ are colors not possessed by $C$, and 
$Z_1$, $Z_2$ are some words. 
While the words $Z_1$, $Z_2$ differ in 
each term of the summation, precisely speaking, we 
use the same letters here for simplicity. 
In the right hand side, each of $S(n,Z_1,j,-n,Z_2)$ 
and $S(n,Z_1,j,-n,Z_2)$ has a single paired color $n$ and 
has multiplicity not bigger than $m(C)$. 

Define a chord diagram $C$ {\it having 
the configuration $\mathcal{F}_l$} ($l=1,2,\ldots)$
to be the one corresponding to a spider 
\[S(c_1, c_2, -c_1, c_3, -c_2, \ldots,  c_l,
-c_{l-1}, X, -c_l, Y),\]
where colors $\pm c_1, \pm c_2, \ldots, \pm c_l$ 
%(some $c_i$ may be negative) 
are mutually distinct, 
and $X$, $Y$ are words (which might be empty) having no 
colors $\pm c_1, \pm c_2, \ldots, \pm c_l$. 
Diagrammatically, a chord diagram $C$ having 
the configuration $\mathcal{F}_l$ is given as in Figure \ref{fig:induction}. 
Now we inductively show that a chord diagram $C$ having 
 the configuration $\mathcal{F}_l$ can be written as 
a linear combination of 
chord diagrams having the configuration $\mathcal{F}_{l+1}$ 
and having 
multiplicities not bigger than $m(C)$ modulo 
brackets unless it is already of the standard form. 

\begin{figure}[htbp]
\begin{center}
\includegraphics[width=0.4\textwidth]{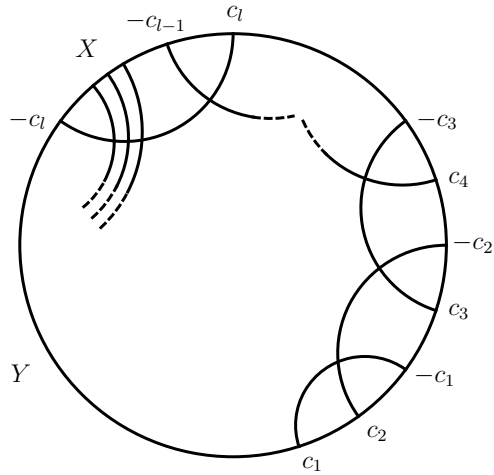}
\end{center}
\caption{The configuration $\mathcal{F}_l$}
\label{fig:induction}
\end{figure}

(The first step) 
By the argument in the third paragraph of this proof, we may assume 
that the chord diagram $C$ has at least one single paired chord 
colored by $\pm c_1$. 
Let $X$ and $Y$ be the regions separated by the single paired chord 
so that the diagram $C$ corresponds to the spider 
$S(c_1,X,-c_1,Y)$. 

If $X$ or $Y$ has no vertices, then $C$ is separable and we are done. 

If $X$ has at least two vertices, we have 
\begin{align*}
S(c_1,X,-c_1,Y) &= [S(c_1,n,-c_1,Y),\ S(-n,X)]\\
&\quad +\sum_{\tiny\begin{subarray}{c}
\mbox{pairing of}\\
\mbox{$X$ and $Y$}
\end{subarray}} \pm S(c_1,n,-c_1,Z_1,-n,Z_2)
\end{align*}
where $n>0$ and $-n$ are new colors as before 
(hereafter we omit these words about the new color $n$). 
Each of the spiders $S(c_1,n,-c_1,Z_1,-n,Z_2)$ 
has the configuration $\mathcal{F}_2$ with $c_2=n$ and multiplicity 
not bigger than $m(C)$. 

If $X$ has only one vertex $v$, then 
$Y$ has at least 
two vertices since $k \ge 3$ and 
the diagram $C$ corresponds to the spider 
$S(c_1,c_v,-c_1,Y)$, where $c_v$ is the color of $v$. 
In this case, there are three possibilities:
\begin{itemize} 
\item[(a)] If $v$ is unpaired, then it is separable. 
\item[(b)] If $v$ is single paired, then $C$ has the configuration 
$\mathcal{F}_2$ with $c_2=c_v$. 
\item[(c)] If $v$ is multiple paired, consider the equality
\begin{align*}
S(c_1,c_v,-c_1,Y) &= [S(c_1,c_v,-c_1,n),\ S(-n,Y)]\\
&\quad +\sum_{\tiny\begin{subarray}{c}
\mbox{color $-c_v$}\\
\mbox{in $Y$}
\end{subarray}} \pm S(-c_1,n,c_1,Z_1,-n,Z_2).
\end{align*}
Each of the spiders $S(-c_1,n,c_1,Z_1,-n,Z_2)$ 
has the configuration $\mathcal{F}_2$ and multiplicity 
less than $m(C)$ since a pair of multiple paired 
vertices colored by $\pm c_v$ was exchanged for single 
paired vertices colored by $\pm n$. 
\end{itemize}

In any case, we have checked that we can proceed to the next step.

\bigskip
(The inductive step) 
Suppose that any chord diagram having 
the configuration $\mathcal{F}_i$ $(i=1,2,3,\ldots,l-1)$ is written as 
a linear combination of chord diagrams 
having the configuration $\mathcal{F}_{i+1}$ and having 
multiplicities not bigger than $m(C)$. Let $C$ be a chord diagram 
having the configuration $\mathcal{F}_l$ as in Figure 
\ref{fig:standard}, where 
$X$ is the region between the vertices colored by $-c_{l-1}$ and $-c_{l}$ 
and $Y$ is the region between the vertices colored by $-c_l$ and $c_1$. 
\begin{itemize}
\item[(I)] Suppose that $X$ has no vertices. If $Y$ has at most one vertex, 
the diagram $C$ is of standard form. Otherwise, $Y$ has at least two 
vertices. Therefore $C$ is separable. 

\item[(II)] Suppose that $X$ has at least two vertices. Then $C$ 
corresponds to the spider 
$S(c_1,c_2,-c_1,c_3,\ldots, c_l, -c_{l-1},X,-c_l,Y)$. 
Consider the equality
\begin{align*}
\hspace{5pt}\lefteqn{S(c_1,c_2,-c_1,c_3,\ldots, c_l, 
-c_{l-1},X,-c_l,Y)}\hspace{15pt}\\
&=[S(c_1,c_2,-c_1,c_3,\ldots, c_l, -c_{l-1},n,-c_l,Y),S(-n,X)]\\
&\quad +\sum_{\tiny\begin{subarray}{c}
\mbox{pairing of}\\
\mbox{$X$ and $Y$}
\end{subarray}} \pm 
S(c_1,c_2,-c_1,c_3,\ldots, c_l, -c_{l-1},n,-c_l,Z_1,-n,Z_2).
\end{align*}
\noindent
Each of the spiders 
$S(c_1,c_2,-c_1,c_3,\ldots, c_l, -c_{l-1},n,-c_l,Z_1,-n,Z_2)$ 
has the configuration $\mathcal{F}_{l+1}$ and multiplicity 
not bigger than $m(C)$. 

\item[(III)] Suppose that $X$ has only one vertex $v$. Let $c_v$ be 
the color of $v$. 

\begin{itemize}
\item[\underline{III-a}] Suppose that $v$ is unpaired. 
If $Y$ has at most one vertex, 
then $C$ is of standard form. Otherwise, $C$ is separable. 

\item[\underline{III-b}] If $v$ is single paired, $C$ has 
the configuration of $\mathcal{F}_{l+1}$. 

\item[\underline{III-c}] If $v$ is multiple paired, 
then $Y$ has at least two vertices. Consider the equality

\begin{align*}
\hspace{5pt}\lefteqn{S(c_1,c_2,-c_1,c_3,\ldots, 
c_l, -c_{l-1},c_v,-c_l,Y)}\hspace{15pt}\\
&=[S(c_1,c_2,-c_1,c_3,\ldots, c_l, -c_{l-1},c_v,-c_l,n),S(-n,Y)]\\
&\quad +\sum_{\tiny\begin{subarray}{c}
\mbox{color $-c_v$}\\
\mbox{in $Y$}\end{subarray}} 
\pm S(c_1,c_2,-c_1,c_3,\ldots, c_l, -c_{l-1},Z_1,-n,Z_2,-c_l,n).
\end{align*}
\noindent
Each of the spiders $S(c_1,c_2,-c_1,c_3,\ldots, c_l, 
-c_{l-1},Z_1,-n,Z_2,-c_l,n)$ 
has the configuration $\mathcal{F}_{l-1}$, namely we have stepped backward. 
However their multiplicity are 
less than $m(C)$ since a pair of multiple paired 
vertices colored by $\pm c_v$ was exchanged for single 
paired vertices colored by $\pm n$. Hence in repeating this rewriting process, 
we meet this case at most $m(C)$ times and 
we can finally go to the next step. 
\end{itemize}
\end{itemize}

Therefore the induction works and we finish the proof. 
\end{proof}

Next we introduce {\it chord slides\/} for chord diagrams to show 
that any chord diagram of the standard form is 
in $\Im [\,\cdot\,,\,\cdot\,]$. 
Hereafter we assume that 
all chord diagrams have no multiple paired vertices. 
Consider spiders having two adjacent vertices 
colored by $i$ and $j$, which might be negative. 
Suppose first that both $i$ and $j$ are single paired. 
Then we have the following equalities: 
\begin{align*}
S(X,i,j,Y,-j,Z,-i)&=\mathrm{sign}(n)[S(X,n,Y,-j,Z,-i),S(i,j,-n)]\\
&\quad +\mathrm{sign}(ni) S(X,n,Y,-j,Z,j,-n)\\
&\quad +\mathrm{sign}(nj) S(X,n,Y,-n,i,Z,-i),\\
S(X,i,j,Y,-i,Z,-j)&=\mathrm{sign}(n)[S(X,n,Y,-i,Z,-j),S(i,j,-n)]\\
&\quad +\mathrm{sign}(ni) S(X,n,Y,-i,Z,-n,i) \\
&\quad +\mathrm{sign}(nj) S(X,n,Y,j,-n,Z,-j),
\end{align*}
where $\mathrm{sign}(m) \in \{\pm 1\}$ denotes the sign of 
an integer $m \neq 0$. 
These equalities are diagrammatically 
expressed (up to sign) as in the first two equalities of 
Figure \ref{fig:chord_slide}, 
which look like ``chord slides to two directions''. 
Next suppose that $i$ is single paired and $j$ is unpaired. 
Then we have
\[S(X,i,j,Y,-i)=\mathrm{sign}(n)[S(X,n,Y,-i),S(i,j,-n)]
+\mathrm{sign}(ni) S(X,n,Y,j,-n).\]
%where we again assume that $n>0$. 
This equality is diagrammatically 
expressed (up to sign) as in the last equality of 
Figure \ref{fig:chord_slide}. Note that in every case of the above, 
the color of the edge on which another chord slides 
changes after a chord slide. 
%\begin{align*}
%S(X,i,j,Y,-i)&=\mathrm{sign}(n)[S(X,n,Y,-i),S(i,j,-n)]
%+\mathrm{sign}(ni) S(X,n,Y,j,-n),\\
%S(X,j,i,Y,-i)&=\mathrm{sign}(n)[S(X,n,Y,-i),S(j,i,-n)]
%+\mathrm{sign}(ni) S(X,n,Y,-n,j), 
%\end{align*}
%where we again assume that $n>0$. These equalities are diagrammatically 
%expressed (up to sign) as in the last two equalities of 
%Figure \ref{fig:chord_slide}. Note that in every case of the above, 
%the color of the edge on which another chord slides 
%changes after a chord slide. 

\begin{figure}[htbp]
\begin{center}
\includegraphics[width=0.6\textwidth]{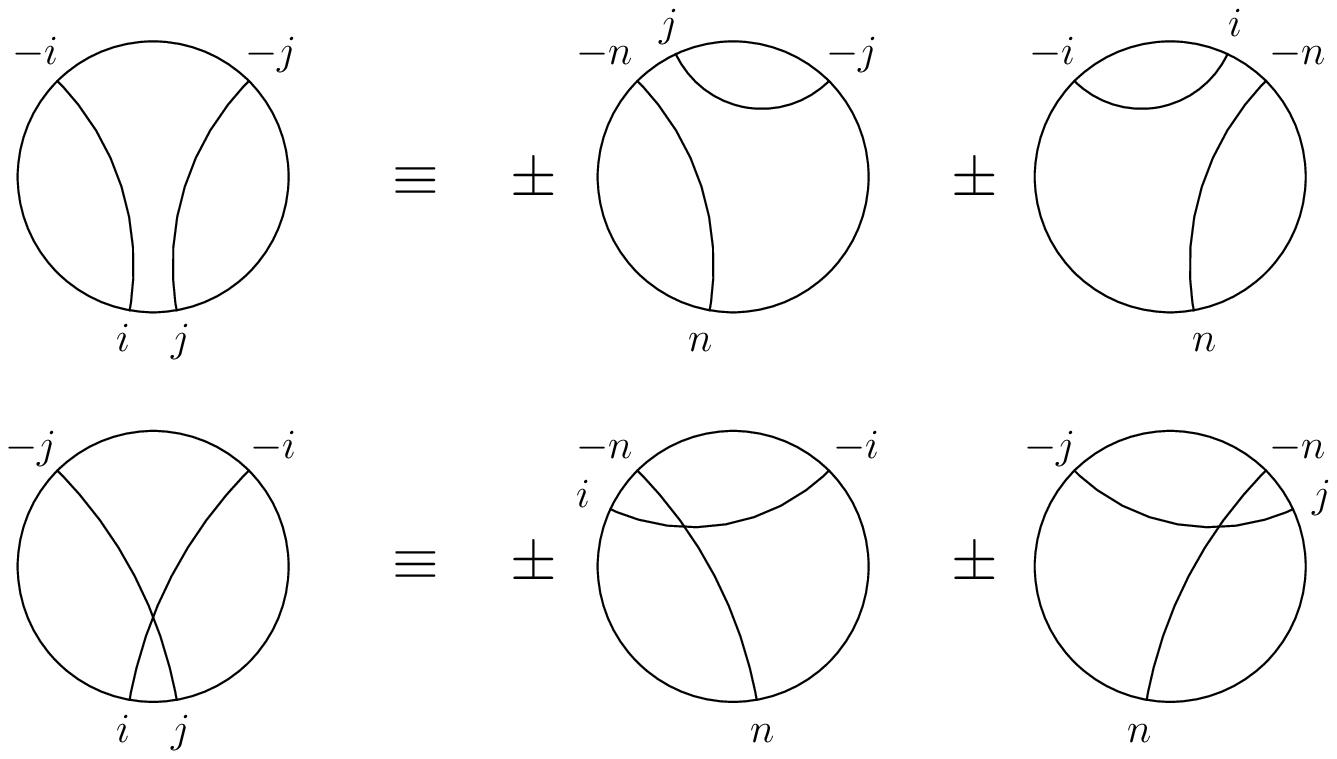}

\bigskip
\includegraphics[width=0.4\textwidth]{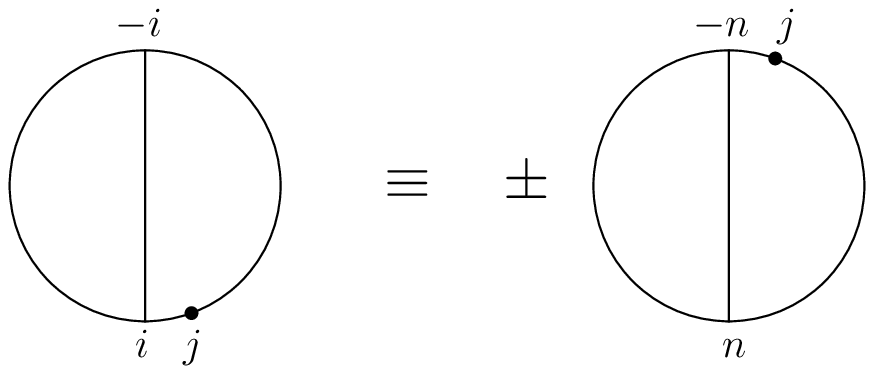}
%\bigskip
%\includegraphics[width=0.4\textwidth]{chord_slide2.eps}
\end{center}
\caption{Chord slides}
\label{fig:chord_slide}
\end{figure}

In using chord slides, the following observation is easy but important. 
Let $\Sigma$ be a surface obtained from a chord diagram of the 
standard form with $l \ge 2$ single paired chords by fattening, 
where we ignore the crossings inside the outer circle. 
The boundary $\partial \Sigma$ of $\Sigma$ contains 
the outer circle and we call 
the other components of $\partial \Sigma$ the {\it inner boundary}. 
\begin{lem}
The inner boundary of $\Sigma$ is connected if $l$ is even, and 
consists of two connected components if $l$ is odd. 
\end{lem}
\begin{proof}%[Sketch of Proof] 
It is easy to see that the statement holds for $l=2$. Then we can 
inductively check that the statement holds for general cases 
by comparing the connection of the boundary before and 
after adding a new chord.
\end{proof}

\begin{proof}[Proof of Theorem $\ref{thm:symp}$ when $k \equiv 0 \pmod 4$] 
There are two patterns of the standard form. 
The first one consists of two unpaired vertices and an even number of 
single paired chords. In this case, we can slide the unpaired vertices 
so that they are adjacent, which is possible because the inner boundary of 
the fattened surface is connected. Then the chord diagram becomes separable. 

The second pattern consists of no unpaired vertex and an odd number of 
single paired chords. To treat this pattern, we 
consider a chord diagram having 
the configuration of the standard form with one more chord $l$ 
intersecting with 
the others at one point as in the left hand side of Figure \ref{fig:cycle1}. 
For such a diagram, we can move the intersecting point by a chord slide 
as shown in the same figure, 
where the second diagram of the right hand side is 
separable. 

\begin{figure}[htbp]
\begin{center}
\includegraphics[width=0.7\textwidth]{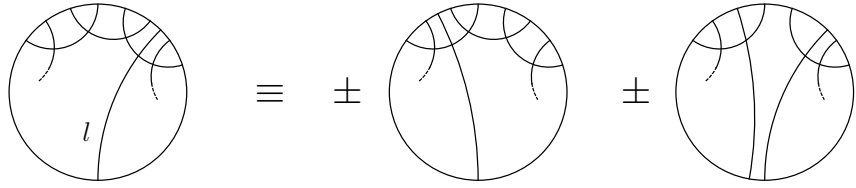}
\end{center}
\caption{Sliding the intersection}
\label{fig:cycle1}
\end{figure}

Now take a chord diagram of the standard form consisting of 
$(k+2)/2$ chords as in 
the left hand side of Figure \ref{fig:cycle2}. 
We may consider it to be a diagram of 
the standard form consisting of 
$k/2$ chords with one more chord colored by $\pm c_1$. 
To this diagram, we apply the chord slide discussed above 
with regarding the chord colored by $\pm c_1$ as $l$. 
By iterating chord slides, we get to the chord diagram of 
the right hand side of Figure \ref{fig:cycle2}, which 
is shown to be in $\Im [\,\cdot\,,\,\cdot\,]$ by considering 
the result of the chord slide at $\ast$ (see also the first line of 
Figure \ref{fig:chord_cyclings}). 
\end{proof}
\begin{figure}[htbp]
\begin{center}
\includegraphics[width=0.5\textwidth]{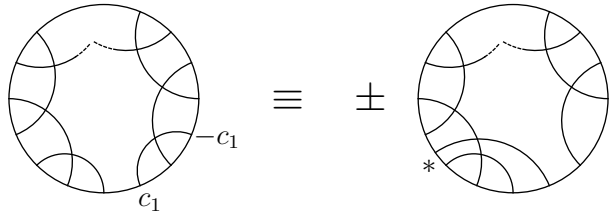}
\end{center}
\caption{The right hand side is the final stage of a chord cycling}
\label{fig:cycle2}
\end{figure}

Hereafter we call the operation used in 
the second pattern of the above (i.e. moving the chord colored by 
$\pm c_1$ from right to left) a {\it chord cycling}. 
\begin{figure}[htbp]
\begin{center}
\includegraphics[width=0.7\textwidth]{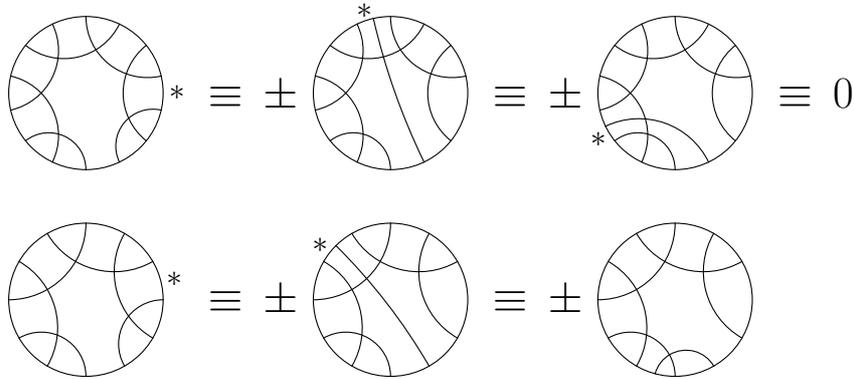}
\end{center}
\caption{Chord cyclings for odd and even numbers of chords}
\label{fig:chord_cyclings}
\end{figure}

\begin{proof}[Proof of Theorem $\ref{thm:symp}$ when $k \equiv 1 \pmod 4$] 
In this case, the standard form consists of a unique unpaired vertex and 
an odd number of single paired chords. Then by a chord cycling 
with ignoring the unpaired vertex, we can slide 
the diagram to a separable one. 
\end{proof}

In the remaining two cases, 
we can apply the same argument as above only to 
chord diagrams of the standard form consisting of two unpaired vertices 
and an odd number of single paired chords, when $k \equiv 2 \pmod 4$. 
Therefore we can finish the proof of Theorem \ref{thm:symp} by considering 
the following two types of chord diagrams (see Figure \ref{fig:typeAB}): 
\begin{itemize}
\item[(a)] chord diagrams of the standard form consisting of 
no unpaired vertices and $(2l+2)$ single paired chords, where 
$k =4l+2$,

\item[(b)] chord diagrams of the standard form consisting of 
one unpaired vertex and $(2l+2)$ single paired chords, where 
$k = 4l+3$.

\begin{figure}[htbp]
\begin{center}
\includegraphics[width=0.7\textwidth]{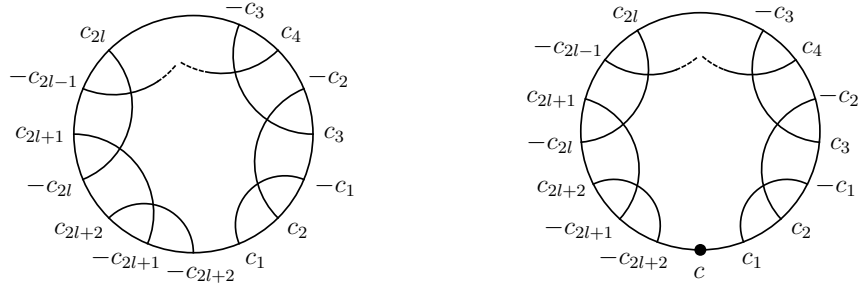}
\end{center}
\caption{Type (a) and Type (b)}
\label{fig:typeAB}
\end{figure}
\end{itemize}
\noindent
For each of them, a chord cycling 
results to another chord diagram of the standard form with distinct colors 
(see the second line of Figure \ref{fig:chord_cyclings}). 
\begin{lem}\label{lem:strong_form}
Under the assumption $g \ge k+3 \ge 6$, we have the following.

$(1)$ Every chord diagram of Type $\mathrm{(a)}$ 
shown in the left of Figure $\ref{fig:typeAB}$ 
is transformed up to sign to the 
one with $c_i=i$ for $i=1,2,\ldots,2l+2$ by chord slides.

$(2)$ Let $C$ be a chord diagram of Type $\mathrm{(b)}$ 
whose unique unpaired vertex 
is colored by $c$ as shown in the right of 
Figure $\ref{fig:typeAB}$. Then for any fixed colors 
$\{d_1, d_2, \ldots, d_{2l+2}\}$ consisting of mutually distinct 
positive integers and 
not including $\pm c$, the diagram $C$ is 
transformed up to sign to the one with $c_i=d_i$ 
for $i=1,2,\ldots,2l+2$ by chord slides.
\end{lem}
\begin{proof}
(1) Let $C$ be a chord diagram of Type (a) 
as in the left of Figure $\ref{fig:typeAB}$. 
We associate this diagram with 
the sequence $[c_1,c_2,\ldots,c_{2l+2}]$ of 
colors. 

By the assumption $g \ge 4l+5$, we have 
\[\{1,2,\ldots,g\}-
\{1,2,\ldots,2l+2, |c_1|, |c_2|, \ldots, |c_{2l+2}|\} 
\neq \emptyset.\]
This means that every time we apply a chord slide, 
we can choose an integer from 
this set as a new color, namely the integer $n$ of the formulas in 
Figure \ref{fig:chord_slide}. 
%For a chord diagram of the standard form, 
%let $[e_1,e_2,\ldots,e_{2l+2}]$ be the sequence of 
%colors obtained 
%by reading the color which appears for the first time
%for each of the chords while looking around the outer circle
%counter-clockwise. 
%For example, the chord diagram $C$ yields the sequence 
%$[c_1,c_2,\ldots,c_{2l+2}]$. 
Taking account of this observation 
%the above discussion, 
we can apply a chord cycling to $C$ so that the resulting 
chord diagram of the 
standard form is associated with the sequence 
\[[c_2, n_3, c_4, n_5, \ldots, n_{2l+1},c_{2l+2},c_1]\]
where $n_3, n_5,\ldots,n_{2l+1} \in 
\{1,2,\ldots,g\}-\{1,2,\ldots,2l+2\}$. 
By iterating chord cyclings, we obtain 
chord diagrams of the standard 
form associated with the sequences
\begin{align*}
&\qquad \, [c_2, n_3, c_4, n_5, \ldots, n_{2l+1},c_{2l+2},c_1] \\
&\longrightarrow 
[n_3, n_4, n_5, \ldots, n_{2l+1},n_{2l+2},c_1,c_2]\\
&\longrightarrow 
[n_4, n'_5, \ldots, n'_{2l+1},n_{2l+2},n_1,c_2,n_3]\\
&\longrightarrow 
[n'_5, n'_6,\ldots, n'_{2l+1},n'_{2l+2},n_1,n_2,n_3,n_4]\\
&\longrightarrow 
[n'_6, 2l+1,n'_8,n''_9, \ldots, n''_{2l+1}, 
n'_{2l+2},n''_1,n_2,n''_3,n_4,n'_5]\\
&\longrightarrow 
[2l+1,2l+2,n''_9, n''_{10},\ldots, n''_{2l+1},n''_{2l+2}, 
n''_1,n''_2,n''_3,n''_4,n'_5,n'_6]\\
&\longrightarrow 
[2l+2,1,n''_{10}, 3 , n''_{12}, 5, \ldots, 2l-4 ,n''_{2l+2}, 
2l-3, n''_2, 2l-2, n''_4,2l-1,n'_6,2l+1]\\
&\longrightarrow 
[1,2,3, \ldots, 2l+1, 2l+2]
\end{align*}
where the positive integers $n_i$, $n'_j$, $n''_k$ are taken 
from $\{1,2,\ldots,g\}-\{1,2,\ldots,2l+2\}$. 
Our claim follows from this. 
Note that the above argument works also for small $l$.

(2) Take a chord diagram of Type (b) 
shown in the right of Figure $\ref{fig:typeAB}$. 
Since the inner boundary is connected, we can slide the unique 
unpaired vertex along all chords so that 
the colors of the other vertices are changed as indicated. This is possible 
because the assumption $g \ge 4l+6$ implies that  
\[\{1,2,\ldots,g\}-
\{|c|,|c_1|,|c_2|,\ldots,|c_{2l+2}|, |d_1|, |d_2|, \ldots, |d_{2l+2}|\} 
\neq \emptyset,\]
which enables us to use an argument similar to (1).
\end{proof}

To show that the chord diagrams specialized in Lemma \ref{lem:strong_form} 
are in $\Im [\,\cdot\,,\,\cdot\,]$, we use the following 
mirror image argument. 
For a spider $S$, we define its {\it mirror} $S^m$ as 
the spider obtained from $S$ by 
sorting its legs in reverse order. 
In terms of chord diagrams, the chord diagram $C(S^m)$ is 
obtained from $C(S)$ by taking its mirror image. 
The following lemma is easily checked. 
\begin{lem}\label{lem:mirror}
For spiders $S_1$ and $S_2$, their bracket 
$[S_1^m,S_2^m]$ is obtained from $[S_1,S_2]$ by taking 
the mirror for each spider in it. 
\end{lem}

%\begin{figure}[htbp]
%\begin{center}
%\includegraphics[width=0.5\textwidth]{cycle2.eps}
%\end{center}
%\caption{A spider $S$ and its mirror $S^m$}
%\label{fig:mirror}
%\end{figure}

\begin{proof}[Proof of Theorem $\ref{thm:symp}$ when $k \equiv 2 \pmod 4$] 
There are two patterns of the standard form. 
The first one consists of two unpaired vertices and an odd number of 
single paired chords. In this case, we can use chord cyclings with 
ignoring the unpaired vertices to show that 
the chord diagram is in $\Im [\,\cdot\,,\,\cdot\,]$ 
as in the cases where $k \equiv 0,1 \pmod 4$. 

The second one is of Type (a), where $k=4l+2$. 
By Lemma \ref{lem:strong_form}, it suffices to 
show that the spider 
\[\widetilde{S}=S(1,2,-1,3,-2,4,\ldots, -2l,2l+2,-(2l+1),-(2l+2))\]
is in $\Im [\,\cdot\,,\,\cdot\,]$. For that, we 
``divide'' the corresponding chord diagram at the center of 
the chain of chords. That is, we consider the equality
\begin{align*}
\widetilde{S}
&=[S(1,2,\ldots, \underline{l+1},-l,\underline{l+2},n), \ 
S(-n,\underline{-(l+1)},l+3,\underline{-(l+2)},\ldots,-(2l+2))]\\
&\quad -S(1,2,\ldots, -(l-1), \underline{l+3,-(l+2),l+4,\ldots,
-(2l+1),-(2l+2),-n},
-l,l+2,n)\\
&\quad -S(-n,-(l+1),l+3, \underline{n,1,2,\ldots,l+1,-l},l+4,\ldots,
-(2l+1),-(2l+2)).
\end{align*}
\noindent
Here we remark that the third term of the right hand side 
is obtained up to sign from the second term by 
taking its mirror and applying the symplectic action 
\begin{align*}
&a_i \longmapsto -b_{2l+3-i}, \quad 
b_i \longmapsto a_{2l+3-i} \quad (i=1,2,\ldots,2l+2),
\\ 
&a_n \longmapsto -b_n, \quad 
b_n \longmapsto a_n.
\end{align*} 

We use Lemma \ref{lem:standard_form} 
to rewrite the second term as the linear combination $P$ of 
chord diagrams of the standard form. 
As for the third term, Lemma \ref{lem:mirror} and the fact that 
the bracket operation is equivariant with respect to the symplectic action 
show that we can rewrite it as the linear combination $Q$ 
obtained from $P$ 
by taking the mirror and applying the symplectic action 
to each chord diagram. 
It follows from Lemma \ref{lem:strong_form} that 
the sum $P+Q$ is rewritten as $2m \widetilde{S}$ 
by some even number $2m$. Therefore we have $\widetilde{S} \equiv 
2m \widetilde{S}$, which implies that $\widetilde{S} \equiv 0$ 
in $H_1 (\ag^+)$. 
\end{proof}
\begin{proof}[Proof of Theorem $\ref{thm:symp}$ when $k \equiv 3 \pmod 4$] 
The standard form consists of a unique unpaired vertex and 
$(2l+2)$ single paired chords, where $k=4l+3$. 
By Lemma \ref{lem:strong_form}, it suffices to 
show that the spider 
\[S(c,d_1,d_2,-d_1,d_3,-d_2,d_4,\ldots, 
-d_{2l},d_{2l+2},-d_{2l+1},-d_{2l+2})\]
is in $\Im [\,\cdot\,,\,\cdot\,]$. 
For that, we can use almost the same argument as 
the case where $k \equiv 2 \pmod 4$ 
by ignoring the unique unpaired vertex. 
Note that the algorithm of Lemma \ref{lem:standard_form} 
keeps the color $c$ of the unpaired vertex. 
\end{proof}

\begin{proof}[Proof of Theorem $\ref{thm:h1}$ $\mathrm{(iii)}$]
If we apply the last split exact sequence in Section \ref{sec:homology} 
to the present case, we have 
\[
H_1 (\ag) \cong 
H_1 (\ag^+)_{\mathfrak{sp}} \oplus H_1 (\ag(0)) = 
H_1 (\ag^+)_{\mathfrak{sp}} \oplus H_1 (\mathfrak{sp}(2g,{\mathbb Q})).
\]
As is well-known that $H_1 (\mathfrak{sp}(2g,{\mathbb Q})) = 0$. 
Hence after taking the limit, we obtain 

%By applying the argument in Section \ref{sec:homology} to 
%the split extension 
%\[0 \longrightarrow \Der^+ (T(H_n)) 
%\longrightarrow \Der (T(H_n))
%\longrightarrow \Der(T(H_n))(0) \longrightarrow 0\]
%of Lie algebras, 
%we obtain an exact sequence 
%\[0 \longrightarrow H_1(\Der^+ (T(H_n)))_{\mathfrak{gl}} 
%\longrightarrow H_1 (\Der (T(H_n)))
%\longrightarrow H_1(\Der(T(H_n))(0)) \longrightarrow 0.\]
%Here we have $H_1 (\Der(T(H_n))(0)) = H_1 (\gl{Z})$ as mentioned above 
%and it is easy to see that 
%the natural pairing $H_n^\ast \otimes H_n \to \Z$ gives an 
%isomorphism $H_1 (\gl{Z}) \cong \Z$. 
%By Theorem \ref{thm:associative}, we have 
%\[H_1(\Der^+ (T(H_n)))_{\mathfrak{gl}} = 
%((H_n^* \otimes H_n^{\otimes 2}) \oplus H_n^{\otimes 2})_{\mathfrak{gl}}
%=0\]
%in the range $n \ge k \ge 2$. 

%By an argument similar to the previous cases, we have 
\[\lim_{g \to \infty} H_1 (\ag) \cong 
\lim_{g \to \infty} H_1 (\ag^+)_{\mathfrak{sp}}.\]
%See also Conant-Vogtmann \cite[Proposition 8]{cv}. 
Now the first author's computation \cite[Theorem 6]{morita_GT} and 
Theorem \ref{thm:symp} show that 
\begin{align*}
\lim_{g \to \infty} H_1 (\ag^+)_{\mathfrak{sp}} &\cong 
\lim_{g \to \infty}  (\ag (1) \oplus 
(\wedge^2 H_\Q/\langle \omega_0 \rangle))_{\mathfrak{sp}}\\
&\cong \lim_{g \to \infty}  (S^3 H_\Q \oplus \wedge^3 H_\Q \oplus 
(\wedge^2 H_\Q/\langle \omega_0 \rangle))_{\mathfrak{sp}}
=0.
\end{align*}
This completes the proof.
\end{proof}

%-----------------moduli4.tex--------2012/3/2, MSS ----
\section{Application to cohomology of moduli spaces of curves}\label{sec:moduli}

In this section, we apply one of our main theorems,
Theorem \ref{thm:h1}, to obtain a new proof of
the vanishing theorem of Harer (Theorem \ref{thm:harerv}).
First we recall the following foundational result
of Harer.
\begin{thm}[Harer {\cite{harer}}]\label{thm:harer}
The virtual cohomological dimension
of $\mathcal{M}_g^m$ is given by
$$
\mathrm{vcd}\,\mathcal{M}_g^m=
\begin{cases}
4g-5 & (g\geq 2, m=0)\\
4g-4+m & (g>0, m>0)\\
m-3  & (g=0)
\end{cases}
$$
so that the rational cohomology group
$$
H^k(\mathbf{M}_g^m;\mathbb{Q})\cong 
H^k(\mathcal{M}_g^m;\mathbb{Q})
$$
vanishes for any $k> \mathrm{vcd}\,\mathcal{M}_g^m$.
\end{thm}
Here we denote by $\mathcal{M}_g^m$
the mapping class
group of $\Sigma_g$ with $m$ distinct marked points
and by $\mathbf{M}_g^m$ the moduli space of curves
of genus $g$ with $m$ distinct marked points.
As is well known, there exists a canonical isomorphism
$$
H^*(\mathbf{M}_g^m;\mathbb{Q})\cong 
H^*(\mathcal{M}_g^m;\mathbb{Q}) \quad (2g-2+m>0).
$$

Now we prove the following result which gives
an alternative proof of the theorem of Harer mentioned above.

\begin{thm}\label{thm:vanishing}
For any $g \ge 2$, the top degree
rational cohomology group
of the moduli space $\mathbf{M}_g^m\ (m=0,1)$ as well as 
the mapping class group $\mathcal{M}_g^m\ (m=0,1)$,
with respect to its virtual cohomological dimension,
vanishes. More precisely, we have
\begin{align*}
&H^{4g-5} (\mathbf{M}_g;\Q) \cong H^{4g-5}(\mathcal{M}_g;\Q)=0
\\
&H^{4g-3} (\mathbf{M}_g^1;\Q) \cong H^{4g-3}(\mathcal{M}_g^1;\Q)=0
\end{align*}
for any $g\geq 2$.
\end{thm}

To prove Theorem \ref{thm:vanishing}, 
we recall the following theorem of Kontsevich 
which is the {\it associative} version 
of the three types of graph (co)homologies 
he presented in \cite{kontsevich1, kontsevich2}.

\begin{thm}[Kontsevich {\cite{kontsevich1, kontsevich2}}]\label{thm:kontsevich}
For $n \ge 1$, there exists an isomorphism 
\[ PH_k\big(\lim_{g \to \infty}\ag\big)_{2n}
\cong 
%PH_k(\mathfrak{sp}(\infty))\oplus
\bigoplus_{\begin{subarray}{c}
2g-2+m=n\\ m>0
\end{subarray}}
H^{2n-k} (\mathbf{M}_g^m;\Q)^{\mathfrak{S}_m}.\]
\end{thm}

\begin{proof}[Proof of Theorem $\ref{thm:vanishing}$]
First we prove the vanishing $H^{4g-3}(\mathcal{M}_g^1;\Q)=0$ for any $g\geq 1$.
By Theorem \ref{thm:h1}, we know that
$\lim_{g\to\infty} H_1(\ag)_{2n}=0$ for any $n$.
If we substitute this in Theorem \ref{thm:kontsevich}, then
we obtain
$$
H^{4g-5+2m}(\mathcal{M}_{g}^{m};\Q)^{\mathfrak{S}_m}=0
\quad \text{for any $m\geq 1$}.
$$
If we put $m=1$, then we can conclude that
$$
H^{4g-3}(\mathcal{M}_{g}^1;\Q)=0
\quad \text{for any $g\geq 1$}.
$$

Next, we deduce $H^{4g-5}(\mathcal{M}_g;\Q)=0\ (g\geq 2)$ from the above. For this, consider the group extension
$$
1\longrightarrow
\pi_1\Sg\longrightarrow 
\mathcal{M}_{g}^1\longrightarrow 
\mathcal{M}_g\longrightarrow
1\quad (g\geq 2)
$$
and let $\{E^{p,q}_r, d^{p,q}_r\}$ denote
the spectral sequence associated to the
above extension for the rational cohomology
group. We have 
$E^{p,q}_2\cong H^p(\mathcal{M}_g;H^q(\pi_1\Sg;\Q))$.
As is well known, there exists a natural isomorphism
$H^q(\pi_1\Sg;\Q)\cong H^q(\Sg;\Q)$ and,
by Theorem \ref{thm:harer},
$H^p(\mathcal{M}_g;\mathcal{H})=0$
for any $p>4g-5$
and for any rational twisted coefficients $\mathcal{H}$.
It follows that the only $E_2$-term,
in total degree $p+q=4g-3$,
which may survive in the $E_{\infty}$ term 
is $E^{4g-5,2}_2\cong H^{4g-5}(\mathcal{M}_g;\Q)$.
On the other hand, it is easy to see that
$$
E_2^{4g-5,2}\cong E_3^{4g-5,2}\cong\cdots\cong
E_\infty^{4g-5,2} = H^{4g-3}(\mathcal{M}_g^1;\Q)=0.
$$
This is a special case of the
fact, proved in \cite{morita87}, that the above
spectral sequence collapses at the $E_2$-term.
We can now conclude that
$H^{4g-5}(\mathcal{M}_g;\Q)=0$ as required.  
\end{proof}

\begin{remark}\label{rem:genus01}
In contrast with the above result, the situation in the cases of genus $0$ and $1$ is 
completely different. According to Getzler \cite{getzler0}, 
the rational cohomology group of top degree
$H^{m-3}(\mathbf{M}_{0}^{m};\Q)$ has dimension $(m-2)!$. 
In \cite{getzler1}, Getzler also determined the $\mathfrak{S}_m$-equivariant Serre characteristic 
for $\mathbf{M}_{1}^{m}$. In particular, the top degree
$\mathfrak{S}_m$-invariant rational cohomology group
$H^{m}(\mathbf{M}_{1}^{m};\Q)^{\mathfrak{S}_m}$ is 
highly non-trivial for infinitely many $m$. 

In \cite{morita_GT}, 
the first author determined the weight $2$ part $H_1(\ag^+)_2$
of the abelianization of $\ag^+$ and by
applying the theorem of Kontsevich 
cited above (Theorem \ref{thm:kontsevich}), 
he constructed a series of cohomology classes in
$H^{4m+1}(\mathbf{M}_1^{4m+1})^{\mathfrak{S}_{4m+1}}$
for $m=1,2,\ldots$.
Then Conant \cite{conant} proved that
these classes are all non-trivial. It would be an interesting
problem to seek for possible special property of these
classes among the whole classes which Getzler determined.
\end{remark}

%-----------------conclusion.tex--------2012/3/2, MSS ----
\section{Concluding remarks}\label{sec:conclusion}

In this section, we make a few remarks concerning the ingredients of this paper.

\begin{remark}\label{rem:h2}
We have been investigating not only the first homology
groups of Lie algebras $\mathfrak{a}_g$ and $\mathfrak{h}_{g,1}$
but also higher homology groups as well. In particular,
we had already a glimpse of considerable difference
between the structures of $H_2(\mathfrak{a}_{g})$ and $H_2(\mathfrak{h}_{g,1})$. 
We will discuss this in a forthcoming paper.
\end{remark}

\begin{remark}\label{rem:other}
The Lie algebra $\mathfrak{a}_g$ appeared in a recent 
work of Enomoto and Satoh \cite{es} and also in
Kawazumi and Kuno \cite{kk}, where they
found certain new roles of this Lie algebra.
We refer to the above cited papers for details.
\end{remark}

%-----------------reference.tex--------2012/3/2, MSS----

\bibliographystyle{amsplain}

\end{document}